\documentclass[10pt,reqno]{amsart}
\usepackage{amscd,amsmath,amsopn,amssymb,amsthm,multicol}
\usepackage{hyperref}
\usepackage[all]{xy}
\usepackage{amscd}
\usepackage{graphicx}

\DeclareMathOperator{\diag}{diag}
\DeclareMathOperator{\Ad}{Ad}
\DeclareMathOperator{\ad}{ad}
\DeclareMathOperator{\Aut}{Aut}

\DeclareMathOperator{\Id}{Id}

\DeclareMathOperator{\Hgt}{ht}

\DeclareMathOperator{\Ric}{Ric}

\DeclareMathOperator{\Span}{span}
\DeclareMathOperator{\rnk}{rk}
\newcommand{\fr}{\mathfrak}
\newcommand{\al}{\alpha}

\newcommand{\be}{\beta}

\newcommand{\bb}{\mathbb}

\DeclareMathOperator{\SO}{SO}
\DeclareMathOperator{\Sp}{Sp}
 \DeclareMathOperator{\SU}{SU}
\DeclareMathOperator{\U}{U}
\DeclareMathOperator{\G}{G}
\DeclareMathOperator{\F}{F}

\DeclareMathOperator{\E}{E}

\newcommand{\thickline}{\noalign{\hrule height 1pt}}
\newtheorem{lemma} {Lemma} [section]
\newtheorem{theorem}[lemma]{Theorem} 
\newtheorem{remark}[lemma] {Remark} 
\newtheorem{prop} [lemma]{Proposition}  
\newtheorem{definition}[lemma] {Definition} 
\newtheorem{corol}[lemma] {Corollary} 
\newtheorem{example}[lemma] {Example}

\begin{document}

\title
{Flag manifolds, symmetric $\fr{t}$-triples and Einstein metrics}

\author{Ioannis Chrysikos}
\address{University of Patras, Department of Mathematics, GR-26500 Rion, Greece}
 \email{xrysikos@master.math.upatras.gr}
 \address{Current adress:  Department of Mathematics and Statistics, Masaryk University, Brno  611 37, Czech Republic}
 \email{chrysikosi@math.muni.cz}
  \medskip

\begin{abstract}
Let $G$ be a compact connected simple Lie group and let $M=G^{\bb{C}}/P=G/K$ be a generalized flag manifold. In this article we   focus on an important  invariant of $G/K$, the so called $\fr{t}$-root system $R_{\fr{t}}$, and we introduce the notion of symmetric $\fr{t}$-triples, that is triples of $\fr{t}$-roots $\xi, \zeta, \eta\in R_{\fr{t}}$ such that $\xi+\eta+\zeta=0$.  We describe their properties and we present an interesting application  on the structure constants of $G/K$, quantities which are straightforward related to the construction of the homogeneous Einstein metric on $G/K$.   Next we classify  symmetric $\fr{t}$-triples for   generalized flag manifolds $G/K$ with  second Betti number $b_{2}(G/K)=1$, and we treat also the case of  full flag manifolds $G/T$, where $T$ is a maximal torus of $G$.  In the last section we construct the homogeneous Einstein equation on flag manifolds $G/K$ with five isotropy summands, determined by  the simple Lie group  $G=\SO(7)$. By solving the corresponding algebraic system we  classify  all $\SO(7)$-invariant (non-isometric) Einstein metrics, and these are the very first results towards the classification of homogeneous Einstein metrics on flag manifolds with five isotropy summands.   



\end{abstract}
 
\maketitle

  \section*{Introduction}
\markboth{Ioannis Chrysikos}{Flag manifolds, symmetric $\fr{t}$-triples and Einstein metrics}

 Let $G$ be a compact,  connected,  simple Lie group with Lie algebra $\fr{g}$ and let   $\fr{g}^{\bb{C}}$ be the complexification of $\fr{g}$. We will denote by $\Ad : G\to\Aut(\fr{g})$ the adjoint representation of $G$, by $\varphi( \cdot , \cdot )$  the Killing form of $\fr{g}$, and by   $( \cdot , \cdot)=-\varphi( \cdot , \cdot )$     the induced $\Ad(G)$-invariant inner product.  Recall that a  generalized  flag manifold is a complex homogeneous space of the form $M=G^{\bb{C}}/P$ where  $G^{\bb{C}}$ is the unique simply connected complex Lie group with Lie algebra  $\fr{g}^{\bb{C}}$ and $P$ is a parabolic subgroup of $G^{\bb{C}}$.  $M$ is   diffeomorphic to the coset $G/K=G/C(S)$, where $C(S)$ is the centralizer of a torus $S\subset G$, and thus $M=\Ad(G)w=\{\Ad(g)w : g\in G\}$ is an adjoint orbit of an element $w\in\fr{g}$.    If   $S$ is a maximal torus in $G$,  say $T$, then $C(T)=T$ and we get the full flag manifold $M=G/T$.  In this case we have the diffeomorphism  $M=G^{\bb{C}}/B\cong G/T$, where $B$ is a Borel subgroup of $G^{\bb{C}}$.     An important   invariant which is closely related to the geometry and the structure of a flag manifold $M=G/K$,  is  the set $R_{\fr{t}}$ of {\it $\fr{t}$-roots}.  These are linear forms obtained by   restricting the set $R_{M}$ of complementary roots of $M$  to the space $\fr{t}$, a real form of  the center of  the Lie subalgebra $\fr{k}^{\bb{C}}\subset\fr{g}^{\bb{C}}$.   They were first introduced by  de Siebenthal  \cite{Sie}, but their current form is due to D. V. Alekseevsky (\cite{Ale1}, \cite{AP}).   To be more specific, $\fr{t}$-roots are the minimal weights of the irreducible submodules of the isotropy representation of $G/K$, and thus they have a fundamental role in    the K\"ahlerian geometry of  $G/K$.  For example,  most $G$-invariant objects on $G/K$, like as Riemannian metrics,  complex structures and K\"ahler-Einstein  metrics, can be expressed in terms of  the $\fr{t}$-root system $R_{\fr{t}}$ (cf. \cite{Ale1}, \cite{AP},   \cite{Arv}).  
  
In this paper we    introduce the notion of  {\it symmetric $\fr{t}$-triples}, that is   triples  of $\fr{t}$-roots $\xi, \zeta, \eta\in  R_{\fr{t}}$ with $\xi+\zeta+\eta=0\in\fr{t}^{*}$.  Due to correspondence between $\fr{t}$-roots $\xi\in  R_{\fr{t}}$ and  non-equivalent irreducible submodules $\fr{m}_{\xi}$ of the $\Ad(K)$-module $\fr{m}^{\bb{C}}$,   symmetric $\fr{t}$-triples are straightforward related  with the structure constants $c_{ij}^{k}$ of $G/K$ which are defined as follows:  We consider an $( \cdot   , \cdot )$-orthogonal reductive decomposition $\fr{g}=\fr{k}\oplus\fr{m}$ of    $\fr{g}=T_{e}G$ and as usual we identify the $\Ad(K)$-invariant subspace  $\fr{m}$ with the tangent space $T_{o}G/K$ of $G/K$ at the identity $o=eK\in G/K$. Now, we asssume that the   direct sum decomposition   $\fr{m}=T_{o}G/K=\fr{m}_1\oplus\cdots\oplus\fr{m}_{s}$ 
  determines an  $( \cdot , \cdot )$-orthogonal   decomposition of $\fr{m}$    into  $s$ pairwise inequivalent irreducible $\Ad(K)$-modules (such a decomposition always exists and    it is given in terms of $\fr{t}$-roots).  We fix   an   $( \cdot , \cdot )$-orthogonal basis $\{e_{\al}\}$ adapted to the  
   decomposition of $\fr{m}$, that is  $e_{\al}\in \fr{m}_{i}$ for some $i$, 
   and $\al<\be$ if $i<j$ (with $e_{\al}\in \fr{m}_{i}$ and $e_{\be}\in\fr{m}_{j}$), and we set  $A_{\al\be}^{\gamma}=([e_{\al}, e_{\be}], e_{\gamma})$ such  
   that $[e_{\al}, e_{\be}]_{\fr{m}}=\sum_{\gamma}A_{\al\be}^{\gamma}X_{\gamma}$, where $[ \ , \ ]_{\fr{m}}$ denotes the $\fr{m}$-component. Then, the structures constanst of $G/K$ with respect to the decomposition $\fr{m}=\oplus_{i=1}^{s}\fr{m}_{i}$, are given by
  \begin{equation}\label{struc}
  c_{ij}^{k}:=\displaystyle {k \brack ij}:=\sum(A_{\al\be}^{\gamma})^{2}=\sum \Big(([e_{\al}, e_{\be}], e_{\gamma})\Big)^{2},
  \end{equation}
    where the sum is taken over all   indices $\al, \be, \gamma$ with $e_{\al}\in \fr{m}_{i}, e_{\be}\in\fr{m}_{j}$,   $e_{\gamma}\in\fr{m}_{k}$, and  $i, j, k\in\{1, \ldots, s\}$  (\cite{Wa2}).   
 In Section \textsection 1 we will  show that symmetric $\fr{t}$-triples   in $\fr{t}^{*}$ are in a bijective correspondence with the non-zero  $c_{ij}^{k}$. Since these quantities   are closely related with the construction of the homogeneous Einstein equation, it  turns out that  symmetric $\fr{t}$-triples   have a key role  in the related theory of  Einstein metrics.

  Recall that a  Riemannian manifold $(M, g)$ is called {\it Einstein}  if it has constant Ricci curvature, i.e. $\Ric_{g}=\lambda  \cdot g$, where  $\lambda \in\bb{R}$ is the so called {\it Einstein constant}.  The Einstein  equation forms a system of  non-linear second order PDEs  and a good understanding of its solutions in the general case seems far from being attained. It is more manageable   when   a Lie group $G$ of isometries acts  on the manifold $M$, via various ways. In the homogeneous case and for a $G$-invariant Riemannian metric,     the Einstein equation  reduces to a system of  algebraic equations which  in some cases can been solved explicitly.  However, even in this case, general existence or non-existence results are difficult to obtained  and are related with  topological notions  (cf. \cite{Wa2}, \cite{SP}, \cite{Ker}). We mention that for a homogeneous Riemannian manifold    $(M=G/K, g)$   of a compact connected (semi)simple Lie group $G$, the Ricci tensor    is expressed in terms of the structure constants $c_{ij}^{k}$,  the parameters $x_{i}$ which define the invariant metric tensor $g$   with respect to the   decomposition  $\fr{m}=\oplus_{i=1}^{s}\fr{m}_{i}$, and the dimensions $d_{i}=\dim_{\bb{R}}\fr{m}_{i}$  for any $i=1, \ldots, s$.   The determination, as well as the calculation of   all the non-zero $c_{ij}^{k}$,  are   non-trivial problems  towards to the formulation  of the   equation  $\Ric_{g}=\lambda  \cdot g$ on $(M=G/K, g)$.

 Recently it has been a lot of progress on homogeneous Einstein metrics on flag manifolds. For example, they have been completely classified   for any flag manifold  $M=G/K$ (of a compact simple Lie group $G$) with two (\cite{Sakane},   \cite{Chry1}, \cite{stauros}),     three (\cite{Kim}, \cite{Arv}, \cite{stauros}),  or four isotropy summands (\cite{Chry2}, \cite{ACS}, \cite{ACS2}).   Homogeneous Einstein metrics on full flag manifolds corresponding to classical Lie groups have been also studied by several authors (cf. \cite{Arv}, \cite{Sak}, \cite{DSN}). Moreover,  in a recent work of the author in collaboration with  A. Arvanitoyeorgos and Y. Sakane   (\cite{ACS3}),   all $\G_2$-invariant Einstein metrics were obtained on the exceptional full flag manifold $\G_2/T$ (a homogeneous space with six isotropy summands).  A further study on invariant Einstein metrics on  $\Sp(n)$-flag manifolds whose $\fr{t}$-root system is of the same form of the full flag manifold $\G_2/T$, namely of {\it $\G_2$-type}, was given in \cite{ACS1}.  Flag manifolds whose isotropy representation decomposes into more than four  isotropy summands are treated also in \cite{CS}, as well as in \cite{Gr}, where M. M. Graev  by applying methods of  aglebraic geometry over complex numbers (Newton polytopes, etc.), presented  a lot of interesting results related with  the number of the complex invariant Einstein metrics on flag spaces.  Another alternative approach to homogeneous Einstein metrics, has been recently established by the author in collaboration with S. Anastassiou (\cite{stauros}); in this paper  the global behaviour of the normalized Ricci flow on the space of $G$-invariant metrics for a flag manifold $G/K$ was studied, and invariant Einstein metrics were obtained explicitly  as the singularities of this flow, located at infity.  This approach seems to give  a better insight on the behaviour of the   Einstein equation and as it has been mentioned in \cite{stauros}, it worths further investigation for a better understanding of its benefits (see also \cite{Grama}).
   
    In this article we classify all invariant Einstein metrics on flag spaces   $G/K$ with five isotropy summands, for the Lie group $G=\SO(7)$.  There are only two such cosets,     defined by the subsets $\Pi_{M_1}=\{\al_1, \al_3\}\subset\Pi$, and $\Pi_{M_2}=\{\al_2, \al_3\}\subset\Pi$, respectively, where $\Pi=\{\al_1, \al_2, \al_3\}$ is a system of simple roots for $\SO(7)$.  Both of them are given  by  $M=\SO(7)/\U(1)\times \U(2)\cong \SO(7)/ \U(1)^{2}\times \SU(2)$,  and as we see in \textsection 3,  these flag manifolds   are isometric (as real manifolds),  and thus we will not distinguish them.  Note that if we paint black both the first two simple roots in the Dynkin diagram of $\SO(7)$, that is $\Pi_{M}=\{\al_1, \al_2\}$, then we obtain also a flag manifold of the form $\SO(7)/ \U(1)^{2}\times \SU(2)\cong \SO(7)/\U(1)^{2}\times \SO(3)$, but it has four isotropy summands (see \cite{Chry2}).  We use  the theory of symmetric $\fr{t}$-triples and we determine all the non-zero structure constants $c_{ij}^{k}$ of $M$, with respect to the decomposition $\fr{m}=\oplus_{i=1}^{5}\fr{m}_{i}$.  For their computation, we  use  a K\"ahler--Einstein metric corresponding to an  $\SO(7)$-invariant complex structure $J$, induced by an invariant ordering $R_{M}^{+}$ on the set of the complementary roots of $M$.  In this way we write down explicitly the Ricci tensor and thus  the  algebraic system which determines the homogeneous Einstein equation.  We prove the following theorem:
   
     \smallskip
    { \bf{Theorem A.}}  
     {\it The flag manifold  $M=\SO(7)/\U(1)\times \U(2)\cong \SO(7)/ \U(1)^{2}\times \SU(2)$, defined by  the  set $\Pi_{M_1}=\{\al_1, \al_3\}\subset\Pi$, or the set $\Pi_{M}=\{\al_2, \al_3\}\subset\Pi$, admits  two   pairs of isometric $\SO(7)$-invariant K\"ahler-Einstein metrics.  There are  also  four  $\SO(7)$-invariant Einstein metrics, which are not K\"ahler with respect to any invariant complex structure on $M$. Two of them  are isometric, thus $M$ admits  three (up to isometry) non-K\"ahler-Einstein metrics.}

  \smallskip
  We have divided the paper into 3 sections. In   \textsection  1 we review  the  Lie theoretic description   of a flag manifold $M=G/K$,  we introduce the notion of symmetric $\fr{t}$-triples and we state their relation with structure constants.  In    \textsection 2, we focus   on  flag manifolds $M=G/K$ with $b_{2}(M)=1$ and we   prove  a structure  theorem    related to the associated $\fr{t}$-root  system $R_{\fr{t}}$ (cf. Theorem \ref{chrysik1}). Based on this result  we    obtain the full classification of symmetric $\fr{t}$-triples for such flag manifolds and we present them  for any case.  Next we extend our study of symmetric $\fr{t}$-triples to full flag manifolds $M=G/T$.  In the  final section \textsection 3,   we give the reader a quick view of the homogeneous Einstein equation and next we  prove Theorem  A.

\markboth{Ioannis Chrysikos}{Flag manifolds, symmetric $\fr{t}$-triples and Einstein metrics}
\section{Flag manifolds and symmetric $\fr{t}$-triples}
\markboth{Ioannis Chrysikos}{Flag manifolds, symmetric $\fr{t}$-triples and Einstein metrics}
Let   $M=G^{\bb{C}}/P=G/C(S)=G/K$ be a  generalized flag manifold of a compact, connected,  simple Lie group,  where $S$ is a torus in $G$. We  begin by  providing an algebraic description of $M$.  We follow the notation of introduction, and we  denote by   $\fr{g}^{\mathbb{C}}$ and $\fr{k}^{\bb{C}}$  the    complexifications of the Lie algebras   of $G$ and $K$, respectively.  We fix a maximal torus  $T$ in $G$ 
which contains the torus $S$ and we denote by $\fr{h}=T_eT$ its Lie algebra, and by $\fr{h}^{\bb{C}}$  the complexification of $\fr{h}$.   Then, $S\subset T\subset C(S)=K$, which means that $T$ is a maximal torus also for the isotropy subgroup $K$ and thus it is $\rnk G=\rnk K$. 
 Let $R\subset(\fr{h}^{\bb{C}})^{*}\backslash\{0\}$   be the root system of $\fr{g}^{\mathbb{C}}$    relative to  the Cartan subalgebra $\fr{h}^{\mathbb{C}}$  and let 
 $
 \fr{g}^{\mathbb{C}}=\fr{h}^{\mathbb{C}}\oplus\sum_{\al\in R}\fr{g}_{\al}^{\mathbb{C}}
$
   be the associated  root space decomposition.   
We consider   vectors $H_{\al}\in\fr{h}^{\bb{C}}$ which are  defined by $\varphi(H, H_{\al})=\al(H)$, for all $H\in\fr{h}^{\mathbb{C}}$, and let $\Pi=\{\al_{1}, \ldots, \al_{\ell}\}$ $(\ell=\dim\fr{h}^{\mathbb{C}})$ be a basis of simple roots for $R$. We will denote by $R^{+}$  the induced ordering.  We set $A_{\al}=E_{\al}+E_{-\al}$ and  $B_{\al}=\sqrt{-1}(E_{\al}-E_{-\al})$, where    $E_{\al}\in\fr{g}_{\al}^{\mathbb{C}}$ $(\al\in R^{+})$ is a Weyl basis of $\fr{g}^{\bb{C}}$ (i.e. $\varphi(E_{\al}, E_{-\al})=-1$  and  $[E_{\al}, E_{-\al}]=-H_{\al}$). Then, the real Lie algebra $\fr{g}$, as  a real form of $\fr{g}^{\bb{C}}$,  is identified with the fixed point set $\fr{g}^{\tau}$ of the conjugation   $\tau : \fr{g}^{\bb{C}}\to\fr{g}^{\bb{C}}$, defined by  $\tau(E_{\al})=E_{-\al}$.  Thus, it is $\fr{g}^{\tau}=\fr{g}=\fr{h}\oplus \sum_{\al\in R^{+}}(\mathbb{R}A_{\al}+\mathbb{R}B_{\al})$.   

Because $\fr{h}^{\bb{C}}\subset\fr{k}^{\bb{C}}\subset\fr{g}^{\bb{C}}$,  there is a closed subsystem $R_{K}$ of $R$ such that 
$\fr{k}^{\bb{C}}=\fr{h}^{\bb{C}}\oplus\sum_{\al\in R_{K}}\fr{g}_{\al}^{\mathbb{C}}$.  Indeed, we can always  find  a subset $\Pi_{K}\subset\Pi$   such that  $R_{K}=R\cap\left\langle\Pi_{K} \right\rangle=\{\be\in R : \be=\sum_{\al_{i}\in\Pi_{K}}k_{i}\al_{i}, \ k_{i}\in\bb{Z}\}$, 
where    $\left\langle\Pi_{K} \right\rangle$  is the space of roots generated by $\Pi_{K}$  with integer coefficients. The complex Lie algebra $\fr{k}^{\mathbb{C}}$  is a maximal reductive subalgebra of $\fr{g}^{\mathbb{C}}$, so we get the decomposition  $\fr{k}^{\mathbb{C}}=Z(\fr{k}^{\mathbb{C}})\oplus\fr{k}^{\mathbb{C}}_{ss},$
where $Z(\fr{k}^{\mathbb{C}})$ is  the  center of $\fr{k}^{\mathbb{C}}$ and  $\fr{k}^{\mathbb{C}}_{ss}=[\fr{k}^{\mathbb{C}}, \fr{k}^{\mathbb{C}}]$ 
is its  semi-simple part.  In particular,   it is $\fr{k}^{\bb{C}}_{ss}=\fr{h}'\oplus\sum_{\al\in R_{K}}\fr{g}_{\al}^{\bb{C}}=\sum_{\al\in\Pi_{K}}\bb{C}H_{\al}\oplus\sum_{\al\in R_{K}}\fr{g}_{\al}^{\bb{C}}$, where $\fr{h}'=\sum_{\al\in\Pi_{K}}\bb{C}H_{\al}\subset\fr{h}^{\bb{C}}$ is the Cartan subalgebra of $\fr{k}^{\mathbb{C}}_{ss}$.   Thus,  $R_{K}$ is  the root system of the semi-simple part  $\fr{k}^{\mathbb{C}}_{ss}$ with a basis of simple roots given by $\Pi_{K}$, i.e. $\dim_{\bb{C}}\fr{h}'=|\Pi_{K}|$, where  $|\Pi_{K}|$ denotes  the cardinality of  $\Pi_{K}$.  The real Lie algebra $\fr{k}$  of $K$  (which is a reductive Lie subalgebra of $\fr{g}$), has the form $\fr{k}=\fr{h}\oplus\sum_{\al\in R_{K}^{+}}(\mathbb{R}A_{\al}+\mathbb{R}B_{\al})$, where  $R_{K}^{+}=R^{+}\cap\left\langle\Pi_{K}\right\rangle$ is the ordering in $R_{K}$ induced by $R^{+}$. The center    $\fr{s} =Z(\fr{k})$ of $\fr{k}$ is given by $\fr{s}=i\fr{t}$, where the subspace $\fr{t}$ is defined  by  (cf. \cite{Ale1}, \cite{Arv}):
\[
\fr{t}=Z(\fr{k}^{\mathbb{C}}) \cap i\fr{h}=\{X\in\fr{h} : \phi(X)=0 \ \mbox{for all} \ \phi\in R_{K}\}.
\] 
  To be more specific, $\fr{t}$ is a real form of the center $Z(\fr{k}^{\mathbb{C}})$, i.e.  $\fr{k}^{\mathbb{C}} =  \fr{t}^{\mathbb{C}}\oplus \fr{k}^{\mathbb{C}}_{ss}$.  Because $\fr{t}\subset\fr{h}\subset\fr{k}$, it is
$\fr{h}=\fr{t}\oplus\fr{t}'$, where $\fr{t}'=\Span\{iH_{\be} : \be\in\Pi_{K}\}$. 
Thus, the Cartan subalgebra $\fr{h}'$ is given by $\fr{h}'=(\fr{t}')^{\bb{C}}$, which means that  $\dim_{\bb{R}}\fr{t}'=\dim_{\bb{C}}\fr{h}'=|\Pi_{K}|$, and so  $\dim_{\bb{R}}\fr{t}=\ell-|\Pi_{K}|$, where $\ell=\rnk \fr{g}^{\bb{C}}   =\dim_{\bb{R}}\fr{h}=\dim_{\bb{R}} T$.  By  \cite[p.~507]{B}, it is    $H^{2}(M; \bb{R})=H^{1}(K; \bb{R})=\fr{t}$, hence
 the second Betti number   of    $M=G/K$ is  equal to  $b_{2}(M)=\dim_{\bb{R}}H^{2}(M; \bb{R})=\dim_{\bb{R}} \fr{t}$, and it  is obtained directly from the   painted Dynkin diagram (see below).

We set $\Pi_{M}:=\Pi\backslash\Pi_{K}$, $R_{M}:=R\backslash R_{K}$ and  $R_{M}^{+}:=R^{+}\backslash R_{K}^+$, such that $\Pi=\Pi_{K}\sqcup\Pi_{M}$, $R=R_{K}\sqcup R_{M}$, and $R^{+}=R^{+}_{K}\sqcup R^{+}_{M}$, respectively.  These splittings  characterize  the flag manifold $M=G/K$.  Roots in  $R_{M}$ are called  {\it complementary roots} and have a key role in   theory which we will describe below.  We recall that an {\it  invariant ordering} $R_{M}^{+}$ in $R_{M}$ is the choise of a subset  $R_{M}^{+}\subset R_{M}$  which satisfies the splitting 
$R=R_{K}\sqcup R_{M}^+\sqcup R_{M}^{-}$, where $R_{M}^{-}=\{-\al: \al\in R_{M}^{+}\}=-R_{M}^{+}$, such that:

(i) \  $\al, \be\in R_{M}^{+}$, $\al+\be\in R_{M}$ $\Longrightarrow$ $\al+\be\in R_{M}^{+}$, 

  (ii)  $\al\in  R_{M}^{+}$,  $\be\in R_{K}^{+}$,  $\al+\be\in R$,  $\Longrightarrow$  $\al+\be\in R_{M}^{+}$.

\noindent We say that $\al>\be$ if and only if $\al-\be\in R_{M}^{+}$. Invariant orderings $R_{M}^{+}\subset R_{M}$ are very useful. For example, let $\fr{g}=\fr{k}\oplus\fr{m}$ be an $( \cdot , \cdot )$-orthogonal reductive decomposition. Then we  see that   the $\Ad(K)$-module $\fr{m}=T_{o}M$ has the form $\fr{m}=T_{o}(G/K)=\sum_{\al\in R_{M}^{+}}(\mathbb{R}A_{\al}+\mathbb{R}B_{\al})$. Moreover, for the complexified version $\fr{m}^{\bb{C}}$ we get the expression $
\fr{m}^{\mathbb{C}}= T_{o}(G/K)^{\mathbb{C}}=\sum_{\al\in R_{M}}\mathbb{C}E_{\al}$. 
 Now, all information contained in  the splitting $\Pi=\Pi_{K}\sqcup\Pi_{M}$ can be presented graphically by the  painted Dynkin diagram of $M=G/K$.
  \begin{definition}\label{pdd}
 Let $\Gamma=\Gamma(\Pi)$ be the Dynkin diagram of the fundamental 
 system $\Pi$.  By painting in  black the nodes 
 of $\Gamma$  corresponding    to  $\Pi_{M}:=\Pi\backslash\Pi_{K}$, we obtain the painted 
 Dynkin diagram (PDD)  of the flag manifold $G/K$. In this diagram
    the subsystem $\Pi_{K}$ is determined by the subdiagram of white roots and each black 
    node  gives rise to one $\U(1)$-component, which their totality forms the center  of $K$.  
 \end{definition}

From now on we fix  a basis $\Pi=\{\al_1,\ldots, \al_r, \phi_1, \ldots, \phi_k\}$      of $R$,  such that $r+k=\ell=\rnk\fr{g}^{\bb{C}}$  and we assume that $\Pi_{K}=\{\phi_1, \ldots, \phi_k\}$ is a basis of the root system 
   $R_{K}$ of $K$.  We set $\Pi_{M}=\Pi\backslash \Pi_{K}=\{\al_{1}, \ldots, \al_{r}\}$.   
 Let $\Lambda_{1}, \ldots, \Lambda_{r}$ be  the fundamental weights    corresponding to the 
simple roots of $\Pi_{M}$, i.e. the linear forms defined by the relations
 $2\varphi(\Lambda_{i}, \al_{j})/ \varphi(\al_{j}, \al_{j})= \delta_{ij}$  and $\varphi(\Lambda_{j}, \phi_{i})= 0$, 
 where $\varphi( \al , \be )$   denotes the inner product on $(\fr{h}^{\mathbb{C}})^{*}$ given by $\varphi(\al, \be)=\varphi(H_{\al}, H_{\be})$, for all $\al, \be\in(\fr{h}^{\mathbb{C}})^{*}$.
 It is well known    that $\{\Lambda_{i} : 1\leq i\leq r\}$ is a basis of the dual space $\fr{t}^{*}$ of $\fr{t}$, thus  $\fr{t}^{*}=\sum_{i=1}^{r}\mathbb{R}\Lambda_{i}$ and $\dim\fr{t}^*=\dim \fr{t}=r$ (\cite{AP}).
Consider   the   restriction map  $\kappa : \fr{h}^{*}\to \fr{t}^{*}$ defined by $\kappa(\al)=\al|_{\fr{t}}$,
and set  $R_{\fr{t}} = \kappa(R)=\kappa(R_{M})$ and $\kappa(R_{K})=0\in\fr{t}^{*}$. 
  The elements of $R_{\fr{t}}$ are called {\it $\fr{t}$-roots}.
  For an invariant ordering $R_{M}^{+}=R^{+}\backslash R_{K}^{+}$ in $R_{M}$, we set $R_{\fr{t}}^{+}=\kappa(R_{M}^{+})$,  $R_{\fr{t}}^{-}=-R_{\fr{t}}^{+}=\{-\xi : \xi\in R_{\fr{t}}^{+}\}$ and we get the splitting $R_{\fr{t}}=R_{\fr{t}}^{+}\sqcup R_{\fr{t}}^{-}$, which defines an ordering in $R_{\fr{t}}$ (cf. \cite{Chry2});     $\fr{t}$-toots $\xi\in R_{\fr{t}}^{+}$ (resp. $\xi\in R_{\fr{t}}^{-}$) are called {\it positive}  (resp. {\it negative}).   A $\fr{t}$-root is called {\it simple} if it is not a sum of two positive $\fr{t}$-roots.    The set $\Pi_{\fr{t}}$ of all simple $\fr{t}$-roots, the so called {\it $\fr{t}$-basis},  is   a basis  of $\fr{t}^*$   in the sense that any $\fr{t}$-root can be written  as a linear combination of its elements with integer coefficients of the same sign.  According to   \cite{AP}, \cite{Chry2}, a $\fr{t}$-basis $\Pi_{\fr{t}}$  is obtained by restricting  the   roots of $\Pi_{M}$ to $\fr{t}$, that is
  $
  \Pi_{\fr{t}}=\{\overline{\al}_{i}=\al_{i}|_{\fr{t}} : \al_{i}\in \Pi_{M}\}.
 $
This allows us to   set up a useful method  to obtain expicitly the set $R_{\fr{t}}$ (see \cite{AA}, \cite{Chry2}). 
 \begin{prop} \label{isotropy} 
  \textnormal{(\cite{Sie}, \cite{Ale1}, \cite{AP})} There exists a bijective correspondence 
 between $\fr{t}$-roots and inequivalent complex irreducible $\ad(\fr{k}^{\mathbb{C}})$-submodules 
 $\fr{m}_{\xi}$ of $\fr{m}^{\mathbb{C}}$, given by:  
\begin{equation}\label{onetoone}
R_{\fr{t}}\ni\xi  \ \leftrightarrow   \ \fr{m}_{\xi} =\sum_{\al\in R_{M}: \kappa(\al)=\xi}\mathbb{C}E_{\al}=\sum_{\al\in R_{M}: \kappa(\al)=\xi}\fr{g}_{\al}^{\bb{C}}.
\end{equation}
Thus $\fr{m}^{\mathbb{C}} = \sum_{\xi\in R_{\fr{t}}} \fr{m}_{\xi}.$   
    Consequently, for   the real $\ad(\fr{k})$-module $\fr{m}=(\fr{m}^{\mathbb{C}})^{\tau}$ we have the following decomposition into   real pairwise inequivalent irreducible   $\ad(\fr{k})$-submodules: $\fr{m} = \sum_{\xi\in R_{\fr{t}}^{+}=\kappa(R_{M}^{+})}(\fr{m}_{\xi}\oplus  \fr{m}_{-\xi})^{\tau}$.
      \end{prop}
In order to study the properties of the decomposition $\fr{m}^{\bb{C}}=\sum_{\xi\in R_{\fr{t}}}\fr{m}_{\xi}$, the following lemma  due to M. Graev is very crucial (for a proof see for example \cite{AA}).

\begin{lemma}\label{lem4}  \textnormal{(Graev)} Let $\xi, \eta, \zeta$ be $\fr{t}$-roots such that $\xi+\eta+\zeta=0$.  Then, there exist roots $\al, \be, \gamma\in R$ with   $\kappa(\al)=\xi, \kappa(\be)=\eta, \kappa(\gamma)=\zeta,$ and such that $\al+\be+\gamma=0$.
\end{lemma}
By using    Graev's lemma  and the properties of the root spaces $\fr{g}_{\al}^{\bb{C}}$  (cf. \cite[p.~ 168]{Hel}), we  obtain that
\begin{corol}\label{rootspaces} Let $M=G/K$ be a flag manifold and let $R_{\fr{t}}$ be its $\fr{t}$-root system. Then 

(1) If $\xi, \eta\in R_{\fr{t}}$ are such that $\xi+\eta \neq 0$, then $(\fr{m}_{\xi}, \fr{m}_{\eta})=0$.

(2) If  $\xi, \eta\in R_{\fr{t}}$ such that $\xi+\eta \neq 0$, $\xi+\eta\in R_{\fr{t}}$, then $[\fr{m}_{\xi}, \fr{m}_{\eta}]\neq 0$.  In particular, it is $[\fr{m}_{\xi}, \fr{m}_{\eta}]=\fr{m}_{\xi+\eta}$.
\end{corol}
\begin{proof}
   (1) Since by assumption it is $\xi+\eta\neq 0$ there is a $\fr{t}$-root, say $\zeta$, such that $\xi+\eta+(-\zeta)=0$.  Then, by Lemma \ref{lem4}, we can find complementary roots $\al, \be, \gamma\in R_{M}$ with $\kappa(\al)=\xi, \kappa(\be)=\eta$, $\kappa(-\gamma)=-\zeta$ and such that $\al+\be+(-\gamma)=0$. But then $\al+\be=\gamma\neq 0$ and thus for the associated root spaces we get  $(\fr{g}_{\al}^{\bb{C}}, \fr{g}_{\be}^{\bb{C}})=0$. Hence, in view of   (\ref{onetoone}) it follows that  $\big(\fr{m}_{\xi}, \fr{m}_{\eta}\big)=0$.

(2) We use again Lemma \ref{lem4} to  find roots   $\al, \be, \gamma\in R_{M}$ with   $\al+\be=\gamma\neq 0$. It means that $\al+\be\in R_{M}\subset R$ and hence $[\fr{g}_{\al}^{\bb{C}}, \fr{g}_{\be}^{\bb{C}}]=\fr{g}_{\al+\be}^{\bb{C}}=\fr{g}_{\gamma}^{\bb{C}}\neq 0$.  It is also $R_{\fr{t}}^{+}\ni\kappa(\al+\be)=\kappa(\al)+\kappa(\be)=\xi+\eta=\zeta$, i.e. $\kappa(\al+\be)\neq 0$, and thus $[\fr{m}_{\xi}, \fr{m}_{\eta}]=\fr{m}_{\zeta}\neq 0$. Since 
$
\fr{m}_{\zeta}=\sum_{\stackrel{\gamma\in R_{M}}{\kappa(\gamma)=\zeta}}\fr{g}_{\gamma}^{\bb{C}}=\sum_{\stackrel{(\al+\be)\in R_{M}}{\kappa(\al+\be)=\xi+\eta}}\fr{g}_{\al+\be}^{\bb{C}},
$
it follows that $[\fr{m}_{\xi}, \fr{m}_{\eta}]=\fr{m}_{\xi+\eta}$.
\end{proof}
Let us now focus on the real  $\Ad(K)$-module $\fr{m} = \sum_{\xi\in R_{\fr{t}}^{+}=\kappa(R_{M}^{+})}(\fr{m}_{\xi}\oplus  \fr{m}_{-\xi})^{\tau}$.  For simplicity we assume that  $R_{\fr{t}}^{+}=\{\xi_1, \ldots, \xi_s\}$.  Then, we have the decomposition $\fr{m}=\oplus_{i=1}^{s}\fr{m}_i$ where  each real irreducible $\ad(\fr{k})$-submodule $\fr{m}_{i}=(\fr{m}_{\xi_{i}}\oplus  \fr{m}_{-\xi_{i}})^{\tau}$ $(1\leq i\leq s)$
corresponding to the positive $\fr{t}$-root $\xi_i$,   is given by
\begin{equation}\label{bas}
\fr{m}_{i}=\sum_{\al\in R_{M}^+ \ :\ \kappa(\al)=\xi_{i}}(\mathbb{R}A_{\al}+\mathbb{R}B_{\al}).
\end{equation} 
  Equation (\ref{bas})  shows that an $( \cdot , \cdot )$-orthogonal  basis of the  component $\fr{m}_{i}$  consists of the vectors $\{A_{\al}= (E_{\al}+E_{-\al}), B_{\al}=i(E_{\al}-E_{-\al})\}$, where  the complementary roots  $\al\in R_{M}^{+}$ are such that $\kappa(\al)=\xi_{i}$, for any $1\leq i\leq s$.   For simplicity, we  denote such a basis by $      \{v_{i}\}=\{A_{\al}, B_{\al} : \al\in R_{M}^{+}, \ \kappa(\al)=\xi_{i}\in R_{\fr{t}}^{+}\}$,   for any $1\leq i\leq s$. Thus, $\dim_{\bb{R}}\fr{m}_{i}=|\{ \pm\al\in R_{M}  : \kappa(\pm\al)=\pm\xi_{i}\}|$.  
  \begin{remark}\label{invord2}\textnormal{By Corollary \ref{rootspaces} (2), it follows that for  any $\al, \be\in R_{M}^{+}$ with $\al+\be\in R$,  it is always $[e_{\al}, e_{\be}]=e_{\al+\be}\in \{v_{i+j}\}$, where   $e_{\al}\in \{v_{i}\}$, $e_{\be}\in\{v_{j}\}$, and $\{v_{i+j}\}=\{A_{\al+\be}, B_{\al+\be} : \al+\be\in R_{M}^{+}, \ \kappa(\al+\be)=\xi_{i}+\xi_{j}\in R_{\fr{t}}^{+}\}$. }
     \end{remark}

 \begin{definition}\label{ttriples} 
  A {\it symmetric $\fr{t}$-triple} in $\fr{t}^{*}$ is a triple $\Xi=(\xi, \eta, \zeta)$ of $\fr{t}$-roots $\xi, \eta, \zeta\in R_{\fr{t}}$ such that $\xi+\eta+\zeta=0$.   
\end{definition}
 
 We mention that a symetric $\fr{t}$-triple $\Xi=(\xi, \eta, \zeta)$ always remains invariant under a permutation  of its components $\xi, \eta$ and $\zeta$, i.e. under the action of the symmetric group $\mathcal{S}_{3}$. 
  Moreover, because in $R_{\fr{t}}$ we  have a polarization $R_{\fr{t}}=R_{\fr{t}}^+\sqcup R_{\fr{t}}^{-}$,  we  associate to any symmetric $\fr{t}$-triple  $\Xi=(\xi, \eta, \zeta)$      a negative one, given by  $-\Xi=(-\xi, -\eta, -\zeta)=-(\xi, \eta, \zeta)$.   Note also  that given a symmetric $\fr{t}$-triple $\Xi=(\xi, \eta, \zeta)$ and an integer $\lambda\in\bb{Z}^{*}$, $\lambda\neq \pm 1$, such that $\lambda\xi, \lambda\eta, \lambda\zeta\in R_{\fr{t}}$, then  $\lambda\xi+\lambda\eta+\lambda\zeta=\lambda(\xi+\eta+\zeta)=0$.  Therefore we can define a new symmetric $\fr{t}$-triple, say $\lambda\Xi$, given by $\lambda\Xi=\lambda(\xi, \eta, \zeta)=(\lambda\xi, \lambda\eta, \lambda\zeta)$.
  \begin{definition}\label{equival}
  Two symmetric $\fr{t}$-triples $\Xi=(\xi, \eta, \zeta)$, $\Xi'=(\xi', \eta', \zeta')$ in $R_{\fr{t}}$ are called equivalent if and only if $\Xi=\pm \Xi'$.
  \end{definition}
\begin{lemma}\label{restriction} Let $M=G/K$ be a generalized flag manifold  and let $R_{\fr{t}}$ be the associated $\fr{t}$-root system. Given a symmetric $\fr{t}$-triple $\Xi=(\xi, \eta, \zeta)$,    the following are true:
 
   (1)  $\Xi$ cannot contain  only positive, or only negative  $\fr{t}$-roots, i.e. it cannot be  $\xi, \eta, \zeta\in R_{\fr{t}}^{+}$, or  $\xi, \eta, \zeta\in R_{\fr{t}}^{-}$.
 
 (2) $\Xi$ cannot contain simultaneously a  $\fr{t}$-root and its negative.

 \end{lemma}  
   
   \begin{proof}
    (1) We will give a proof     for positive $\fr{t}$-roots,  and for  negative $\fr{t}$-roots it is similar.  Let $\Xi=(\xi, \eta, \zeta)$ be a symmetric $\fr{t}$-triple    such that $\xi, \eta, \zeta\in R_{\fr{t}}^{+}$.  By Lemma  \ref{lem4}, there   exist roots $\al, \be, \gamma\in R_{M}$ with   $\kappa(\al)=\xi, \kappa(\be)=\eta$ and $\kappa(\gamma)=\zeta$, such that $\al+\be+\gamma=0$.  But since $\xi, \eta, \zeta\in R_{\fr{t}}^{+}$ and $R_{\fr{t}}^{+}=\kappa(R_{M}^{+})$ it must be  $\al, \be, \gamma\in R_{M}^{+}$. Since $\al+\be+\gamma=0$ we get that $\al+\be=-\gamma\in R_{M}^{-}$, a contradiction.
   
  (2) Let  $\Xi=(\xi, \eta, \zeta)$ be  a symmetric $\fr{t}$-triple such that $\eta=-\xi$.  Then, 
because $\kappa(\pm\al)=\pm\xi$ for some $\al\in R_{M}$, it must be $\be=-\al$  in the relation $\al+\be+\gamma=0$, where  the complementary roots  $\be, \gamma\in R_{M}$ are such that  $\kappa(\be)=\eta$ and $\kappa(\gamma)=\zeta$.  Thus  we conclude that $\gamma=0$, which is   a contradiction since $\gamma\in R_{M}$.  Similarly   all possible combinations  are treated.   
     \end{proof}
   \begin{corol}\label{new}
 Let $M=G/K$ be a generalized flag manifold of a compact simple Lie group $G$  and let $R_{\fr{t}}$ be the  associated $\fr{t}$-root  system.  Assume that  $\fr{m}=\oplus_{i=1}^{s}\fr{m}_{i}$ is an $( \cdot , \cdot )$-orthogonal decomposition of $\fr{m}$ into pairwise inequivalent irreducible $\ad(\fr{k})$-modules, and let $\xi_{i}, \xi_{j}, \xi_{k}\in R_{\fr{t}}$ be the $\fr{t}$-roots associated to the components $\fr{m}_{i}$, $\fr{m}_{j}$ and $\fr{m}_{k}$, respectively,  Then,  $c_{ij}^{k}=\displaystyle{ k \brack ij}\neq 0$, if and only if  $(\xi_{i}, \xi_{j}, \xi_{k})$ is a symmetric $\fr{t}$-triple, i.e.  $\xi_{i}+ \xi_{j}+\xi_{k}=0$.  
   \end{corol}
  \begin{proof}
 If $c_{ij}^{k}\neq 0$,  for some indices $i, j, k\in\{1, \ldots, s\}$, then $\varphi([\fr{m}_{i}, \fr{m}_{j}], \fr{m}_{k})\neq 0$. We will prove that $\xi_{i}+\xi_{j}+\xi_{k}=0$.  For any   $\fr{m}_{i}=(\fr{m}_{\xi_{i}}\oplus\fr{m}_{-\xi_{i}})^{\tau}$ we choose an orthogonal basis $\{u_{i}\}=\{A_{\al}, B_{\al} : \al\in R_{M}^{+}, \kappa(\al)=\xi_{i}\in R_{\fr{t}}^{+}\}$. Then, from  Remark \ref{invord2} we know that  for any  $e_{\al}\in \{v_{i}\}$ and $e_{\be}\in\{v_{j}\}$, it is     $[e_{\al}, e_{\be}]=e_{\al+\be}\in\{v_{i+j}\}$,  unless   $\al+\be=0$.  Moreover,    by  Corollary \ref{rootspaces} (1), it follows that for some $e_{\gamma}\in\{v_{k}\}$ it is $\phi([e_{\al}, e_{\be}], e_{\gamma})=0$,  unless   $\al+\be+\gamma=0$.  However, it is $\varphi([\fr{m}_{i}, \fr{m}_{j}], \fr{m}_{k})\neq 0$ and thus it must be $\al+\be\neq 0$ and    $\al+\be+\gamma=0$.  Thus  
  	 $
  	 \xi_{i}+\xi_{j}+\xi_{k}=\kappa(\al)+\kappa(\be)+\kappa(\gamma)=\kappa(\al+\be+\gamma)=0.
  $   The converse  is a trivial consequence of Lemma \ref{lem4}. 
  \end{proof}

    \markboth{Ioannis Chrysikos}{Flag manifolds, symmetric $\fr{t}$-triples and Einstein metrics}
      \section{Symmetric $\fr{t}$-triples for certain classes of flag manifolds}
    \markboth{Ioannis Chrysikos}{Flag manifolds, symmetric $\fr{t}$-triples and Einstein metrics}
 Important classes of generalized flag manifolds $M=G/K$ for which one can gives  general expressions of symmetric $\fr{t}$-triples are for instance, the flag manifolds $G/K$ with second  Betti number $b_{2}(G/K)=1$, and the  full flag manifolds $G/T$ (which are such that $b_{2}(G/T)=\ell=\rnk G$). We start with the first family.

 \subsection{Symmetric $\fr{t}$-triples on flag manifolds with second Betti number equal to $1$} 
 Let $M=G/K$ be a generalized flag manifold of a compact simple Lie group $G$, defined by a subset $\Pi_{M}=\{\al_{i}\}\subset\Pi$. Then, it is $\dim_{\bb{R}}\fr{t}=|\Pi_{M}|=1$, thus $M$ is such that $b_{2}(G/K)=1$. In particular, any flag manifold with second Betti number equal to 1, is defined in this way.  Recall that the {\it height} of a simple root $\al_{i}\in\Pi$  $(i=1, \ldots, \ell)$, is  the  positive integer   $m_{i}$ in 
the expression of the highest root $\widetilde{\al}=\sum_{k=1}^{\ell}m_{k}\al_{k}$ in terms of simple roots.  We will denote by $\Hgt : \Pi\to\bb{Z}^{+}$  
      the function which associates to each simple root its height, that is $\Hgt(\al_{i})=m_{i}$.
 
      \begin{theorem}\label{chrysik1}
     Let $M=G/K$ be a generalized flag manifold of a compact simple Lie group $G$,    defined by a subset $\Pi_{M}=\{\al_{i}\}$  where the fixed simple root $\al_{i}\in\Pi$ is such that $\Hgt(\al_{i})=r\geq 2$.  Let $R_{\fr{t}}$ be the associated $\fr{t}$-root system and $\Pi_{\fr{t}}$ the corresponding $\fr{t}$-basis.  Then, given $\xi\in R_{\fr{t}}^{+}$ such that $\xi\notin\Pi_{\fr{t}}$, it is $\xi-\overline{\al}_{i}\in R_{\fr{t}}^{+}$, where $\overline{\al}_{i}=\kappa(\al_{i})=\al_{i}|_{\fr{t}}\in\Pi_{\fr{t}}$. 
      \end{theorem}
      \begin{proof}
     According to \cite{Chry2}, a $\fr{t}$-basis is given by  $\Pi_{\fr{t}}=\{\overline{\al}_{i}\}$ and thus $\fr{t}^{*}=\bb{R}\overline{\al}_{i}$, where   $\overline{\al}_{i}=\kappa(\al_{i})=\al_{i}|_{\fr{t}}$.   However, by assumption it is $\Hgt(\al_{i})=r\geq 2$, thus   $|R_{\fr{t}}^{+}|=r\geq 2$.  Indeed, let $\al=\sum_{j=1}^{\ell}k_{j}\al_{j}\in R^{+}$, where the non-negative coefficients $k_{j}$ are such that $k_{j}\leq m_{j}=\Hgt(\al_{j})$ for any $j=1, \ldots, \ell$.  Then we have
     $\kappa(\al)=k_{i}\overline{\al}_{i}$, with $1\leq k_{i}\leq m_{i}=r$, which means that  
    $ R_{\fr{t}}^{+}=\{\overline{\al}_{i}, 2\overline{\al}_{i}, \ldots, r\overline{\al}_{i}\}$.
   Thus we obtain an irreducible decomposition  $\fr{m}=\fr{m}_1\oplus\cdots\oplus\fr{m}_{r}$, where each summand $\fr{m}_{k}$ is given by (\ref{bas})  for any $1\leq k\leq r$.      Let now $\xi\in R_{\fr{t}}^{+}$ such that $\xi\neq \overline{\al}_{i}$.  Then, it is $\xi=p\overline{\al}_{i}$ with $2\leq p\leq r$ and so 
    $
     \xi-\overline{\al}_{i}=p\overline{\al}_{i}-\overline{\al}_{i}=(p-1)\overline{\al}_{i}\in R_{\fr{t}}^{+},
    $
   which proves our claim.   For $p=2$ we have $\xi-\overline{\al}_{i}=\overline{\al}_{i}\in R_{\fr{t}}^{+}$, while for $p=r$, it is $\xi-\overline{\al}_{i}=(r-1)\overline{\al}_{i}\in R_{\fr{t}}^{+}$, since $r\overline{\al}_{i}\in R_{\fr{t}}^{+}$.
      \end{proof}
     Theorem \ref{chrysik1} generalizes   the following well known theorem of root systems theory,  to the  $\fr{t}$-root system $R_{\fr{t}}$ corresponding to a flag manifold $M=G/K$ with $b_{2}(G/K)=1$.
      \begin{lemma}\label{carter}\textnormal{(\cite[p.~ 460]{Hel})}
      Let $R$ be the root system of a complex semisimple Lie algebra. Choose a basis $\Pi=\{\al_1, \ldots, \al_{\ell}\}$ and let $R^{+}$ be the induced ordering in $R$. If $\al\in R^{+}$ such that $\al\notin\Pi$, then there exists some $\al_{i}\in\Pi$ such that $\al-\al_{i}\in R^{+}$.
      \end{lemma}


      \begin{corol}\label{betti1}
      Let $M=G/K$ be a  flag manifold with $b_{2}(M)=1$ and let $R_{\fr{t}}$ be its $\fr{t}$-root system. If $|R_{\fr{t}}^{+}|\geq 2$,  then any $\fr{t}$-root $\xi\in R_{\fr{t}}$  belongs to a symmetric $\fr{t}$-triple, i.e. we can find $\zeta, \eta\in R_{\fr{t}}$ such that $\xi+\zeta+\eta=0$.    
      \end{corol}
      \begin{proof}
Assume that $M$ is defined by the subset $\Pi_{M}=\{\al_{i}\}$, for some fixed $i\in\{1, \ldots, \ell\}$.  Let $\xi\in R_{\fr{t}}^{+}$. Then from Theorem \ref{chrysik1} it is $\xi-\overline{\al}_{i}\in R_{\fr{t}}^{+}$, where $\overline{\al}_{i}$ is the unique element of the associated $\fr{t}$-basis $\Pi_{\fr{t}}$. But then $(\overline{\al}_{i}, \xi-\overline{\al}_{i}, -\xi)$ is a symmetric $\fr{t}$-triple since $\overline{\al}_{i}+ (\xi-\overline{\al}_{i}) + (-\xi)=0$ and $\overline{\al}_{i}, \xi-\overline{\al}_{i}\in R_{\fr{t}}^{+}$, and $-\xi\in R_{\fr{t}}^{-}$.
  Thus,  a general form of symmetric $\fr{t}$-triples on $M=G/K$   is 
   \begin{equation}\label{b21}
     \pm(\overline{\al}_{i}, \xi-\overline{\al}_{i}, -\xi), \quad \xi\in R_{\fr{t}}^{+}, 
     \end{equation}
    We call a symmetric $\fr{t}$-triple of the form (\ref{b21}),  a {\it symmetric $\fr{t}$-triple  of  Type A}.  
    Note that if $\xi\in\Pi_{\fr{t}}\subset R_{\fr{t}}^{+}$, i.e. $\xi=\overline{\al}_{i}$, then since by assumption it is $b_{2}(M)=1$ and $|R_{\fr{t}}^{+}|\geq 2$, the set $R_{\fr{t}}^{+}$ must contain   the $\fr{t}$-roots $\xi=\overline{\al}_{i}$ and $2\xi=2\overline{\al}_{i}$.  Thus a symmetric $\fr{t}$-triple  which contains $\overline{\al}_{i}$ is given by 
   $\pm \big(\overline{\al}_{i}, \overline{\al}_{i}, -2\overline{\al}_{i}\big)$.  But this is also a symmetric $\fr{t}$-triple of Type A since it is obtained from (\ref{b21}) for $\xi=2\overline{\al}_{i}$.  Now,   if $\xi\in R_{\fr{t}}^{-}$, then $-\xi\in R_{\fr{t}}^{+}$, and thus in order to obtain the symmetric $\fr{t}$-triple which contains  $-\xi$ we have to replace in   (\ref{b21})  $\xi$ by $-\xi$. In this case the desired  $\fr{t}$-triple is given by 
  $
   \pm(\overline{\al}_{i}, -\xi-\overline{\al}_{i}, \xi),
 $
   which proves our claim.
       \end{proof}
    
    In the following, for a given flag manifold $G/K$ with $\fr{t}$-root system $R_{\fr{t}}$,  we will denote by $\mathcal{S}$ the set of inquivalent symmetric $\fr{t}$-triples in $R_{\fr{t}}\subset\fr{t}^{*}$.
    \begin{theorem}\label{height1}
    The only flag manifolds $M=G/K$ of a compact simple Lie group $G$, for which the set $\mathcal{S}$ is empty, are the compact isotropy irreducible Hermitian symmetric spaces.
    \end{theorem}
 \begin{proof}
 Let $M=G/K$ be a generalized flag manifold of a compact simple Lie group $G$ and let $\Pi=\{\al_1, \ldots, \al_{\ell}\}$ be a system of simple roots for $G$. Assume that $M$ is  defined by a subset $\Pi_{M}=\{\al_{i}\}\subset\Pi$, for some fixed $i\in\{1, \ldots, \ell\}$, such that  $\Hgt(\al_{i})=1$.
    Then $M=G/K$ is an isotropy irreducible Hermitian symmetric space of compact type  (cf.  \cite{Hel}, \cite{CS}).  Indeed, a $\fr{t}$-basis is given by $\Pi_{\fr{t}}=\{\overline{\al}_{i}\}$ where   $\overline{\al}_{i}=\kappa(\al_{i})=\al_{i}|_{\fr{t}}$.  Since $\Hgt(\al_{i})=1$, it follows that $R_{\fr{t}}=\{\pm\overline{\al}_{i}\}$ and $|R_{\fr{t}}^{+}|=1$, which means that $M$ is isotropy irreducible.  Due to the form of $R_{\fr{t}}$ it is obvious that we cannot construct a symmetric $\fr{t}$-triple which contains some of the $\fr{t}$-roots $\overline{\al}_{i}$ or $-\overline{\al}_{i}$.  Thus $\mathcal{S}=\emptyset$.
    \end{proof}
  
   Due to Theorem \ref{height1}, it is now clear   the condition $|R_{\fr{t}}^{+}|\geq 2$ has been assumed in  Theorem \ref{chrysik1} and Corollary \ref{betti1}.    In particular,  any flag manifold $M=G/K$ of a compact simple Lie group $G$ which is not an isotropy irreducible Hermitian symmetric space, is such that  $|R_{\fr{t}}^{+}|\geq 2$, and thus  there exists at least one symmetric $\fr{t}$-triple, of the form $(\xi, \zeta, -(\xi+\zeta))$, where $\xi\neq\zeta\in R_{\fr{t}}^{+}$ (and thus, at least one non-zero structure constant). 
  
  \begin{remark}\label{maxheight}
  \textnormal{The isotropy representation of   a flag manifold  $M=G/K$ of a compact  simple Lie group $G$ with $b_{2}(M)=1$, may have at most six isotropy summands.  This natural constraint comes from  the form of the highest root   $\widetilde{\al}$ corresponding to a  complex semisimple Lie algebra  (cf. \cite[p. ~477]{Hel}).  For example, a flag manifold $M=G/K$ with $b_{2}(M)=1$ and $G\in\{\SU(\ell+1), \SO(2\ell+1), \Sp(\ell), \SO(2\ell)\}$, could be either an isotropy irreducible Hermitian symmetric space  of compact type, or a generalized flag manifold with two isotropy summands. This is because the heights of the simple roots  for classical simple Lie groups are not  greater than two.    Moreover, the only simple Lie group whose   root system contains  simple  roots  $\al_{i}$ with $\Hgt(\al_{i})=5$, or $\Hgt(\al_{i})=6$ is $G=\E_8$, and only for this Lie group we can determine  flag manifolds $M=G/K$ with $b_{2}(M)=1$ and  five or six isotropy summands.  For more details we refer to \cite{CS} (see also Example \ref{height5}).} 
  \end{remark}

    \begin{remark}\label{bet1}
    \textnormal{Since the $\fr{t}$-root system of flag manifold $M=G/K$ as in Corrollary \ref{betti1}, is given by  $R_{\fr{t}}^{+}=\{\overline{\al}_{i}, 2\overline{\al}_{i}, \ldots, r\overline{\al}_{i}\}$, $r\geq 2$, another symmetric $\fr{t}$-triple  is given by
   \begin{equation}\label{b212}
    \pm(p\overline{\al}_{i}, q\overline{\al}_{i}, -(p+q)\overline{\al}_{i}), 
   \end{equation}
   where $2\leq p, q\leq r$ such that $4\leq p+q\leq r$.  Since $p, q\geq 2$, it is obvious that these symmetric $\fr{t}$-triples are inequivalent to the symmetric $\fr{t}$-triples of Type A, and thus we will call them     {\it symmetric $\fr{t}$-triples of  Type B}.  Note that symmetric $\fr{t}$-triples of Type A and B are the only possible symmetric $\fr{t}$-triples which one can construct for a flag manifold $M=G/K$ with $b_{2}(M)=1$, that is $\mathcal{S}=\{(\xi, \eta, \zeta) \ \text{of Type A or B}\}$.  By Remark \ref{maxheight}  it follows that  symmetric $\fr{t}$-triples of Type B     only exist for the values $r=4, 5, 6$, and they are given by:}
  \end{remark}
    
  \begin{center}
  {\bf Table 1.} \ {\small Symmetric $\fr{t}$-triples of Type B}
 \end{center}
 \smallskip
  \begin{center}
     \begin{tabular}{ccc}
     \hline 
         $r=4$ & $r=5$ & $r=6$ \\
    \thickline 
  $\pm(2\overline{\al}_{i}, 2\overline{\al}_{i}, -4\overline{\al}_{i})$ & $\pm(2\overline{\al}_{i}, 2\overline{\al}_{i}, -4\overline{\al}_{i})$ & $\pm(2\overline{\al}_{i}, 2\overline{\al}_{i}, -4\overline{\al}_{i})$ \\
   & $\pm(2\overline{\al}_{i}, 3\overline{\al}_{i}, -5\overline{\al}_{i})$ & $\pm(2\overline{\al}_{i}, 3\overline{\al}_{i}, -5\overline{\al}_{i})$ \\
   & & $\pm(2\overline{\al}_{i}, 4\overline{\al}_{i}, -6\overline{\al}_{i})$ \\
   & & $\pm(3\overline{\al}_{i}, 3\overline{\al}_{i}, -6\overline{\al}_{i})$ \\ 
 \hline
 \end{tabular}
 \end{center}

\medskip

\noindent By Corollary \ref{betti1}, Remarks \ref{maxheight}, \ref{bet1}  and Table 1 we also conclude that 
\begin{theorem}\label{number}
Let $M=G/K$ be a generalized flag manifold of a compact simple Lie group with $b_{2}(G/K)=1$  and $|R_{\fr{t}}^{+}|=r\geq 2$.  Let $N$ denote  the number of non equivalent symmetric $\fr{t}$-triples  on $\fr{t}^{*}$, that is $N=|\mathcal{S}|$.  Then  $N\geq |R_{\fr{t}}^{+}|-|\Pi_{\fr{t}}|=r-1$.  The exact number $N$ is given in the following table:
\end{theorem}
 
  \begin{center}
  {\bf Table 2.} \ {\small The number $N=|\mathcal{S}|$ of inequivalent symmetric $\fr{t}$-triples on $\fr{t}^{*}$ for $M=G/K$ with $b_{2}(G/K)=1$}
 \end{center}
 \smallskip
  \begin{center}
     \begin{tabular}{r|c|c|c|c|c}
         $r=|R_{\fr{t}}^{+}|$   & $r=2$ & $r=3$ & $r=4$ & $r=5$ & $r=6$ \\
        \thickline
        $N$ & $1$ & $2$ & $4$ & $6$ &  $9$ \\    
  \end{tabular}
 \end{center}
  
    \begin{proof}
    From Corollary \ref{betti1} we obtain at least $|R_{\fr{t}}^{+}|-|\Pi_{\fr{t}}|=r-1$ symmetric $\fr{t}$-triples of Type A given by (\ref{b21}). Thus $N\geq r-1$.  In particular, for $r=2$ or  $3$, one can determine only  symmetric $\fr{t}$-triples of Type A on $\fr{t}^{*}$, and thus the exact number $N$ of these triples is $N=1$  and $N=2$, respectively.  If $4\leq r \leq 6$, by Remark \ref{bet1} we know  that on $\fr{t}^{*}$,  there exist also   the symmetric $\fr{t}$-triples of Type B. In any case, the exact number $n$ of symmetric $\fr{t}$-triples of Type B is obtained from  Table 1. Thus  for $4\leq r \leq 6$  the number $N$ is given by $N=r-1+n$, where  $n=1, 2, 4$ for  $r=4, 5, 6$, respectively.    
    \end{proof}

          
      Next we present for all  flag manifolds $M=G/K$ (of a compact simple Lie group $G$)  with $b_{2}(G/K)=1$, the associated symmetric $\fr{t}$-triples and the non-zero structure constants $c_{ij}^{k}$.  We denote by  $\Pi=\{\al_1, \ldots, \al_{\ell}\}$   a basis of simple roots for $G$, adapted to the choice of $G$. 
      \begin{example}\label{height2}
      \textnormal{Let $M=G/K$  defined by a subset   $\Pi_{M}=\{\al_{i}\}\subset\Pi$ such that $\Hgt(\al_{i})=2$. Then $R_{\fr{t}}=\{\pm\overline{\al}_{i}, \pm2\overline{\al}_{i}\}$, $|R_{\fr{t}}^{+}|=2$ and thus   we obtain the decomposition $\fr{m}=\fr{m}_1\oplus\fr{m}_2$, where the components $\fr{m}_{k}$ are determined  by (\ref{bas}). These spaces have been classified in \cite{Chry1} or \cite{Sakane}.  By Theorem \ref{number} we know that there exists only one symmetric $\fr{t}$-triple of Type A,  given by $\pm(\overline{\al}_{i}, \overline{\al}_{i}, -2\overline{\al}_{i})$. Thus by Corollary \ref{new}     the only non-zero structure constant of $M=G/K$ is the triple $c_{11}^{2}$ and its symmetries. The number $c_{11}^{2}$ was calculated in \cite{Chry1} in terms of the dimensions $d_{i}=\dim_{\bb{R}}\fr{m}_{i}$ $(i=1, 2)$.}
      \end{example}

      \begin{example}\label{height3}
      \textnormal{Let $M=G/K$   defined by a subset  $\Pi_{M}=\{\al_{i}\}\subset\Pi$ such that  $\Hgt(\al_{i})=3$. Then $R_{\fr{t}}=\{\pm\overline{\al}_{i}, \pm2\overline{\al}_{i}, \pm 3\overline{\al}_{i}\}$ and $|R_{\fr{t}}^{+}|=3$.  Thus, $\fr{m}=\fr{m}_1\oplus\fr{m}_2\oplus\fr{m}_3$, where the summands $\fr{m}_{k}$ are  determined  by (\ref{bas}). These flag manifolds have been classified in \cite{Kim}, but we  note that  they do not exhaust all flag manifolds with three isotropy summands (see Remark \ref{3tropy}). By applying (\ref{b21}), we  find   two  symmetric $\fr{t}$-triples of Type A given by $\pm(\overline{\al}_{i}, \overline{\al}_{i}, -2\overline{\al}_{i})$ and $\pm(\overline{\al}_{i}, 2\overline{\al}_{i}, -3\overline{\al}_{i})$.  These are the only symmetric $\fr{t}$-triples on $\fr{t}^{*}$.   Thus the non-zero triples are $c_{11}^{2}, c_{12}^{3}$ and their symmetries.  The  values of these triples can be found  in \cite{stauros}.}
      \end{example}

       \begin{example}\label{height4}
      \textnormal{Let $M=G/K$   defined by a subset  $\Pi_{M}=\{\al_{i}\}\subset\Pi$ such that $\Hgt(\al_{i})=4$. Then $R_{\fr{t}}=\{\pm\overline{\al}_{i}, \pm2\overline{\al}_{i}, \pm 3\overline{\al}_{i}, \pm 4\overline{\al}_{i}\}$ and $|R_{\fr{t}}^{+}|=4$. Thus  $\fr{m}=\fr{m}_1\oplus\fr{m}_2\oplus\fr{m}_3\oplus\fr{m}_4$, where the components $\fr{m}_{k}$ are  given by (\ref{bas}). These spaces have been classified in \cite{Chry2}, but  they do not   exhaust all flag manifolds with four isotropy summands (see Remark \ref{4isotropy}).  By applying (\ref{b21}) we find three  symmetric $\fr{t}$-triples of Type A, given as follows: 
     $\pm(\overline{\al}_{i}, \overline{\al}_{i}, -2\overline{\al}_{i})$,  $\pm(\overline{\al}_{i}, 2\overline{\al}_{i}, -3\overline{\al}_{i})$, $ \pm(\overline{\al}_{i}, 3\overline{\al}_{i}, -4\overline{\al}_{i})$.
           Also, as we have seen in Table 1, by applying (\ref{b212}) we obtain  a symmetric $\fr{t}$-triple of Type B, given by $\pm(2\overline{\al}_{i}, 2\overline{\al}_{i}, -4\overline{\al}_{i})$.
    Thus    the non-zero structure constants of $M=G/K$ are $c_{11}^{2},  c_{12}^{3},   c_{13}^{4}, c_{22}^{4}$,  and their symmetries (cf. \cite{Chry2}).}
      \end{example}

      \begin{example}\label{height5}
      \textnormal{Let $M=G/K$   defined by a subset  $\Pi_{M}=\{\al_{i}\}\subset\Pi$ such that $\Hgt(\al_{i})=5$.  By Remark \ref{maxheight} we know that such a choise exists only for $G=\E_8$. In particular, a basis of simple roots for the root system of $G_8$ can be choosen such that the highest root  $\widetilde{\al}=2\al_{1}+3\al_{2}+4\al_{3}+5\al_{4}+6\al_{5}+4\al_{6}+2\al_{7}+3\al_{8}$ (cf. \cite{AA}, \cite{Chry2}). Thus, only the simple root $\al_4$ is such that $\Hgt(\al_4)=5$ and by setting $\Pi_{M}=\{\al_4\}$ we obtain the  painted  Dynkin diagram }
      \medskip
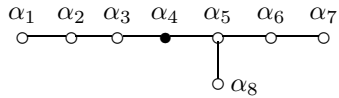
\[
\begin{picture}(160,25)(-5, -10)
\put(15, 9.5){\circle{4 }}
\put(15, 18){\makebox(0,0){$\al_1$}}
\put(17, 10){\line(1,0){14}}
\put(33.5, 9.5){\circle{4 }}
\put(33.5, 18){\makebox(0,0){$\al_2$}}
\put(35, 10){\line(1,0){13.6}}
 \put(51, 9.5){\circle{4 }}
 \put(51, 18){\makebox(0,0){$\al_3$}}
\put(53,10){\line(1,0){14}}
\put(69,9.5){\circle*{4 }}
\put(69, 18){\makebox(0,0){$\al_4$}}
\put(89,-8){\circle{4}}
\put(99, -9.5){\makebox(0,0){$\al_8$}}
\put(89,-6){\line(0,1){14}}
\put(71,10){\line(1,0){16}}
\put(89,9.5){\circle{4 }}
\put(89, 18){\makebox(0,0){$\al_5$}}
\put(90.7,10){\line(1,0){16}}
\put(109,9.5){\circle{4 }}
\put(109, 18){\makebox(0,0){$\al_6$}}
\put(111,10){\line(1,0){16}}
\put(129,9.5){\circle{4 }}
\put(129, 18){\makebox(0,0){$\al_7$}}
\end{picture}
 \]
 \textnormal{which defines the flag manifold  $M=G/K=\E_8/\U(1)\times \SU(4)\times \SU(5)$. Here we have the decomposition $\fr{m}=\fr{m}_1\oplus\cdots\oplus\fr{m}_5$, since $R_{\fr{t}}^{+}=\{\overline{\al}_{4}, 2\overline{\al}_{4}, 3\overline{\al}_{4}, 4\overline{\al}_{4}, 5\overline{\al}_{4}\}$.  Theorem  \ref{number}  states  that $|\mathcal{S}|=6$.      Indeed,   we obtain the following symmetric $\fr{t}$-triples:
  \begin{eqnarray*}
  && \overline{\al}_{4}+\overline{\al}_{4}+(-2\overline{\al}_{4})=0, \ \overline{\al}_{4}+2\overline{\al}_{4}+(-3\overline{\al}_{4})=0, \ \overline{\al}_{4}+3\overline{\al}_{4}+(-4\overline{\al}_{4})=0,\\
  && \overline{\al}_{4}+4\overline{\al}_{4}+(-5\overline{\al}_{4})=0, \ 2\overline{\al}_{4}+2\overline{\al}_{4}+(-4\overline{\al}_{4})=0, \  2\overline{\al}_{4}+3\overline{\al}_{4}+(-5\overline{\al}_{4})=0.
  \end{eqnarray*}
  Thus the non-zero structure constants are $c_{11}^{2}$, $c_{12}^{3}$, $c_{13}^{4}$, $c_{14}^{5}$, $c_{22}^{4}$, $c_{23}^{5}$, and their symmetries.
  Notice that if we set $\Pi_{M}=\{\al_5\}$, then we obtain the painted Dynkin diagram 
  
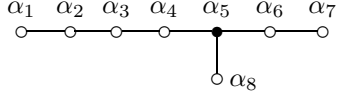
\[
\begin{picture}(160,25)(-5, -10)
\put(15, 9.5){\circle{4 }}
\put(15, 18){\makebox(0,0){$\al_1$}}
\put(17, 10){\line(1,0){14}}
\put(33.5, 9.5){\circle{4 }}
\put(33.5, 18){\makebox(0,0){$\al_2$}}
\put(35, 10){\line(1,0){13.6}}
 \put(51, 9.5){\circle{4 }}
 \put(51, 18){\makebox(0,0){$\al_3$}}
\put(53,10){\line(1,0){14}}
\put(69,9.5){\circle{4 }}
\put(69, 18){\makebox(0,0){$\al_4$}}
\put(89,-8){\circle{4}}
\put(99, -9.5){\makebox(0,0){$\al_8$}}
\put(89,-6){\line(0,1){14}}
\put(71,10){\line(1,0){16}}
\put(89,9.5){\circle*{4 }}
\put(89, 18){\makebox(0,0){$\al_5$}}
\put(90.7,10){\line(1,0){16}}
\put(109,9.5){\circle{4 }}
\put(109, 18){\makebox(0,0){$\al_6$}}
\put(111,10){\line(1,0){16}}
\put(129,9.5){\circle{4 }}
\put(129, 18){\makebox(0,0){$\al_7$}}
\end{picture}
 \]
 which defines the flag manifold $M=G/K=\E_8/\U(1)\times \SU(2)\times \SU(3)\times \SU(5)$.  Since $\Hgt(\al_{5})=6$, this is the only flag manifold (of a compact simple Lie group) with $b_{2}(M)=1$ and $\fr{m}=\fr{m}_1\oplus\cdots\oplus\fr{m}_{6}$.  For this case,  Theorem  \ref{number}  states  that $|\mathcal{S}|=9$.   Indeed, once can easily determine the following symmetric $\fr{t}$-triples:
  \begin{eqnarray*}
  && \overline{\al}_{5}+\overline{\al}_{5}+(-2\overline{\al}_{5})=0, \ \overline{\al}_{5}+2\overline{\al}_{5}+(-3\overline{\al}_{5})=0, \ \overline{\al}_{5}+3\overline{\al}_{5}+(-4\overline{\al}_{5})=0,\\
  && \overline{\al}_{5}+4\overline{\al}_{5}+(-5\overline{\al}_{5})=0, \
  \overline{\al}_{5}+5\overline{\al}_{5}+(-6\overline{\al}_{5})=0, \
   2\overline{\al}_{5}+2\overline{\al}_{5}+(-4\overline{\al}_{5})=0, \\
&&  2\overline{\al}_{5}+3\overline{\al}_{5}+(-5\overline{\al}_{5})=0, \  2\overline{\al}_{5}+4\overline{\al}_{5}+(-6\overline{\al}_{5})=0, \  3\overline{\al}_{5}+3\overline{\al}_{5}+(-6\overline{\al}_{5})=0.
  \end{eqnarray*} Thus, the non-zero structure constants are $c_{11}^{2}$, $c_{12}^{3}$, $c_{13}^{4}$, $c_{14}^{5}$, $c_{15}^{6}$, $c_{22}^{4}$, $c_{23}^{5}$, $c_{24}^{6}$, $c_{33}^{6}$ and their symmetries.}
    \end{example}

\begin{remark}\label{3tropy}
\textnormal{In \cite{Kim}, it was proved that flag manifolds with three isotropy summands are also defined by setting $\Pi_{M}=\{\al_{i},  \al_{j} : i\neq j\}$ such that   $\Hgt(\al_{i})=\Hgt(\al_{j})=1$. For these spaces it was shown that there is  only one non-zero stucture constant, namely $c_{12}^{3}$ and its symmetries. Indeed, for such a flag manifold it is $\dim_{\bb{R}}\fr{t}=2$ and thus $b_{2}(M)=2$. A $\fr{t}$-basis is given by $\Pi_{\fr{t}}=\{\overline{\al}_{i}=\al_{i}|_{\fr{t}}, \overline{\al}_{j}=\al_{j}|_{\fr{t}} : i\neq j\}$, thus by choosing a positive root $\al=\sum_{k=1}^{\ell}c_{k}\al_{k}\in R^{+}_{M}$  we conclude that any positive $\fr{t}$-root is given by $\xi=\kappa(\al)=c_{i}\overline{\al}_{i}+c_{j}\overline{\al}_{j}$, where $0\leq c_{1}, c_{j}\leq 1$, since by assumption it is $\Hgt(\al_{i})=\Hgt(\al_{j})=1$.  Note that we cannot have simultaneously $c_{i}=c_{j}=0$, because then $\al\in R_{K}$.  Therefore  the $\fr{t}$-root system is given by $R_{\fr{t}}=\{\pm\overline{\al}_{i}, \pm\overline{\al}_{j}, \pm(\overline{\al}_{i}+\overline{\al}_{j})\}$, and so $\fr{m}=\fr{m}_1\oplus\fr{m}_2\oplus\fr{m}_3$, where the summands $\fr{m}_{k}$  are defined according  to (\ref{bas}).  Now  it is obvious that the only symmetric  $\fr{t}$-triple is given by $\pm\big( \overline{\al}_{i},  \overline{\al}_{j}, -(\overline{\al}_{i}+\overline{\al}_{j})\big)$
and thus  $c_{12}^{3}\neq 0$.}
 \end{remark}

 \begin{remark}\label{4isotropy}   \textnormal{ 
  In \cite{Chry2}, the author  proved that flag manifolds with four isotropy summands are also defined by sets of the form $\Pi_{M}=\{\al_{i}, \al_{j} : i\neq j\}$ such that $\Hgt(\al_{i})=1$, $\Hgt(\al_{j})=2$, or $\Hgt(\al_{i})=2$, $\Hgt(\al_{j})=1$.   However, since subsets $\Pi_{M}$ of the last form  also detemine  flag manifolds $M=G/K$ with five isotropy summands, this  correspondence is not one-to-one.   Indeed, a $\fr{t}$-basis  is given by $\Pi_{\fr{t}}=\{\overline{\al}_{i}, \overline{\al}_{j} : i\neq j\}$, where $\overline{\al}_{i}=\al_{i}|_{\fr{t}}$ and $\overline{\al}_{j}=\al_{j}|_{\fr{t}}$. If we assume for example that $\Hgt(\al_{i})=1$, and $\Hgt(\al_{j})=2$, then    given a root $\al=\sum_{k=1}^{\ell}c_{k}\al_{k}\in R^{+}$  we have $\xi=\kappa(\al)=c_{i}\overline{\al}_{i}+c_{j}\overline{\al}_{j}$, where $0\leq c_{1}\leq 1$ and $0\leq  c_{j}\leq 2$. Thus, one can determine  at most five positive $\fr{t}$-roots, given by $\overline{\al}_{i}$, $\overline{\al}_{j}$, $\overline{\al}_{i}+\overline{\al}_{j}$, $\overline{\al}_{i}+2\overline{\al}_{j}$, $2\overline{\al}_{j}$.
    The existence of the last $\fr{t}$-root  in the above sequence depends on the root system of $G$, and more particularly on the existence of a   root $\al =\sum_{k=1}^{\ell}c_{k}\al_{k}\in R^{+}$ such that $c_{i}=0$ and $c_{j}=2$.    Such roots appear if $G\in\{\SO(2\ell+1), \SO(2\ell), \E_6, \E_7\}$, and the associated flag manifolds $M=G/K$ are such that $|R_{\fr{t}}^{+}|=5$ (cf. \cite[Prop. 6]{Chry2}). The subsets $\Pi_{M}=\{\al_{i}, \al_{j} : i\neq j\}$ which define exactly four positive $\fr{t}$-roots can be found in \cite[Table 4]{Chry2}.   In order to study the symmetric $\fr{t}$-triples for the associated flag manifold $M=G/K$ with $\fr{m}=\fr{m}_1\oplus\fr{m}_2\oplus\fr{m}_3\oplus\fr{m}_4$,  we  set for simplicity
  $
  \Pi_{M}^{A}=\{\al_{i}, \al_{j} : \Hgt(\al_{i})=1, \ \Hgt(\al_{j})=2\}$ and $\Pi_{M}^{B}=\{\al_{i}, \al_{j} : \Hgt(\al_{i})=2, \ \Hgt(\al_{j})=1\}.
 $
   Then, the  corresponding $\fr{t}$-root  systems are given by
   $
     R_{\fr{t}}^{A}=\{\pm\overline{\al}_{i}, \pm\overline{\al}_{j}, \pm(\overline{\al}_{i}+\overline{\al}_{j}), \pm(\overline{\al}_{i}+2\overline{\al}_{j})\}$ and $R_{\fr{t}}^{B}=\{\pm\overline{\al}_{i}, \pm\overline{\al}_{j}, \pm(\overline{\al}_{i}+\overline{\al}_{j}), \pm(2\overline{\al}_{i}+\overline{\al}_{j})\}$, respectively.
    For $R_{\fr{t}}^{A}$ we find the symmetric $\fr{t}$-triples  $\pm\big( \overline{\al}_{i},  \overline{\al}_{j}, -(\overline{\al}_{i}+\overline{\al}_{j})\big)$, and  $\pm\big(\overline{\al}_{j}, \overline{\al}_{i}+\overline{\al}_{j}, -(\overline{\al}_{i}+2\overline{\al}_{j})\big)$.
         So, the only non-zero  structure constants of the corresponding  flag manifolds, are $c_{12}^{3}$, $c_{23}^{4}$ and their symmetries.  For  $R_{\fr{t}}^{B}$  we get two symmetric $\fr{t}$-triples  given by $\pm\big( \overline{\al}_{i},  \overline{\al}_{j}, -(\overline{\al}_{i}+\overline{\al}_{j})\big)$, and $\pm\big(\overline{\al}_{i}, \overline{\al}_{i}+\overline{\al}_{j}, -(2\overline{\al}_{i}+\overline{\al}_{j})\big)$.
     Hence,  the only non-zero  structure constants of the corresponding flag manifolds, are $c_{12}^{3}$, $c_{13}^{4}$ and their symmetries. The explicit values of these triples are given in \cite{Chry2}.}
 \end{remark}

  \subsection{Symmetric $\fr{t}$-triples for full flag manifolds $M=G/T$}
 Let us now extend our study of symmetric $\fr{t}$-triples on   full flag manifolds $M=G/T$, where  $T$ is maximal torus of a compact connected simple Lie group $G$.     Such a space  is obtained by painting black all nodes in   the Dynkin diagram of $G$, that is  $\Pi_{K}=\emptyset$ and  $\Pi_{M}=\Pi=\{\al_1, \ldots, \al_\ell\}$.  It follows that $R_{K}=\emptyset$ and  $R=R_{M}$, i.e. the set of complementary roots of $M$ is identified with the root system of $G$, and hence the associated $\fr{t}$-root system $R_{\fr{t}}$ of $G/T$ has the properties of  a root system.    It is obvious that $\fr{t}=T_{e}(T)$, thus $\dim_{\bb{R}}\fr{t}=\dim_{\bb{R}} T=\ell=\rnk G$ and $b_{2}(G/T)=\ell$. Also $\fr{t}^{*}=\sum_{i=1}^{\ell}\mathbb{R}\Lambda_{i}$, where $\Lambda_{i}$ are the fundamental weights corresponding to $\Pi$, and a  $\fr{t}$-basis $\Pi_{\fr{t}}$ is given by $\Pi_{\fr{t}}=\{\overline{\al}_{p}=\al_{p}|_{\fr{t}} : 1\leq p\leq \ell\}=\{\overline{\al}_1, \ldots, \overline{\al}_{\ell}\}$, i.e. $|\Pi_{\fr{t}}|=|\Pi_{M}|=|\Pi|=\ell$.  Since $R^{+}_{M}=R^{+}$, it is also $R^{+}_{\fr{t}}=\kappa(R^{+})$ and $|R_{\fr{t}}^{+}|=|R^{+}|$.  
 \begin{prop}\label{GT}
 For a full flag manifold  $M=G/T$ of a compact simple Lie group $G$, there is a bijective correspondence between roots and $\fr{t}$-roots.
  \end{prop}   
  \begin{proof}
 The kernel of the linear map $\kappa :\fr{h}^{*}\to\fr{t}^*$  is given by ${\rm Ker}\kappa = \{\al\in\fr{h}^* : \kappa(\al)=0\} = R_{K}\cup\{0\}$,  and thus  in general $\kappa$ is not an injection. However, for a full flag manifold $M=G/T$ it is $R_{K}=\emptyset$ and ${\rm Ker} \kappa=\{0\}$, thus we obtain the desired correspondence.
  \end{proof}    
  
  \begin{prop}\label{numberfull}
 Let  $M=G/T$ be a full flag manifold  of a compact simple Lie group $G$.  Then the isotropy representation of $M$ decomposes into a direct sum of 2-dimensional pairwise inequivalent irreducible $T$-submodules $\fr{m}_{\al}$.  The number of these submodules is equal to   $|R^{+}|$.  
 \end{prop}
  \begin{proof} 
   Let $\fr{g}=\fr{t}\oplus\fr{m}$ be a reductuve decomposition associated to $G/T$. Then
  $
 \fr{m}=\sum_{\al\in R^{+}}(\fr{m}_{\kappa(\al)}\oplus\fr{m}_{-\kappa(\al)})^{\tau}=\sum_{\al\in R^{+}}\fr{m}_{\al},
$
 where we have set for simplicity $\fr{m}_{\al}=(\fr{m}_{\kappa(\al)}\oplus\fr{m}_{-\kappa(\al)})^{\tau}$, for any $\al\in R^{+}$.  Each $\fr{m}_{\al}$   is a  real irreducible $\Ad(T)$-submodule, which  does not depend on the choise  of the simple roots of $R$.  From (\ref{bas}) it follows that $\fr{m}_{\al}=\mathbb{R}A_{\al}+\mathbb{R}B_{\al}=\mathbb{R}(E_{\al}+E_{-\al})+\mathbb{R}i(E_{\al}-E_{-\al})$,
  and this completes the proof.
  \end{proof}
 In Table 1, following for example \cite{Hel}, we give for any full flag manifold $G/T$ of a compact simple Lie group $G$  the exact number  of the corresponding isotropy summands.
  
 \begin{center}
 {\bf Table 3.} \ {\small The number of the isotropy summands for $G/T$}
 \end{center}
 \smallskip
{\footnotesize{ \begin{center}
 \begin{tabular}{llll}
 \hline 
 $\mbox{Simple  Lie group} \ G$ & Full flag manifold   $G/T$ & $|R|=|R_{\fr{t}}|$ & $\fr{m}=\oplus_{i=1}^{s}\fr{m}_{i}$ \\
  \thickline
  $\SU(\ell+1), \ \ell\geq 1$ & $\SU(\ell+1)/T$ & $\ell(\ell+1)$ & $s=\ell(\ell+1)/2$ \\
  $\SO(2\ell+1), \ \ell\geq 2$ & $\SO(2\ell+1)/T$ & $2\ell^{2}$ & $s=\ell^{2}$ \\
  $\Sp(\ell), \ \ell\geq 2$ & $\Sp(\ell)/T$ & $2\ell^{2}$ & $s=\ell^{2}$ \\
  $\SO(2\ell), \ \ell\geq 3$ & $\SO(2\ell)/T$ & $2\ell(\ell-1)$ & $s=\ell(\ell-1)$ \\
  $\G_2$ & $\G_2/T$ & $12$ & $s=6$ \\
  $\F_4$ & $\F_4/T$ & $48$ & $s=24$ \\
  $\E_6$ & $\E_6/T$ & $72$ & $s=36$ \\
  $\E_7$ & $\E_7/T$ & $126$ & $s=63$ \\
  $\E_8$ & $\E_8/T$ & $240$ & $s=120$ \\
  \hline 
  \end{tabular} 
 \end{center}}}
 
 \medskip
 The full flags $\SU(\ell+1)/T$ and $\SO(2\ell)/T$  for $\ell=1$ and $\ell=2$, respectively, give rise to  the complex projective line $\bb{C}P^{1}=\SU(2)/\U(1)\cong \SO(4)/\U(2)$, which is an isotropy irreducible Hermitian symmetric space and thus from Theorem \ref{height1}  we have $\mathcal{S}=\emptyset$. As we will see in the following, this is the only full flag manifold for which one cannot define symmetric $\fr{t}$-triples, and thus all structure constants are zero. Next, we will prove that full flag manifolds give rise to a second class of  flag manifolds, for which  a generalization of Lemma \ref{carter} holds. 
   \begin{theorem}\label{chrysik2}
  Let $M=G/T$ be a full flag manifold of a compact simple Lie group $G$, $R_{\fr{t}}$ the associated $\fr{t}$-root system, and $\Pi_{\fr{t}}$ the corresponding $\fr{t}$-basis. Then,  given $\xi\in R_{\fr{t}}^{+}$ such that $\xi\notin\Pi_{\fr{t}}$ we can   find at least one $\al_{p}\in\Pi$ such that $\xi-\overline{\al}_{p}\in R_{\fr{t}}^{+}$, where  $\overline{\al}_{p}=\kappa(\al_{p})=\al_{p}|_{\fr{t}}\in\Pi_{\fr{t}}$.
  \end{theorem}
  \begin{proof}
  We have seen that a $\fr{t}$-basis is given by $\Pi_{\fr{t}}=\{\overline{\al}_{p}=\al_{p}|_{\fr{t}} : 1\leq p\leq \ell\}$.  Assume that the result is false, that is $\xi-\overline{\al}_{p}\notin R_{\fr{t}}^{+}$, for any $\overline{\al}_{p}=\al_{p}|_{\fr{t}}\in\Pi_{\fr{t}}$.   We will show that  $\xi-\overline{\al}_{p}\notin R_{\fr{t}}^{-}$, so $\xi-\overline{\al}_{p}\notin R_{\fr{t}}$. In contrary, we assume that $\xi-\overline{\al}_{p}\in R_{\fr{t}}^{-}$.  Thus $\overline{\al}_{p}-\xi\in R_{\fr{t}}^{+}$. But then $\overline{\al}_{p}=(\overline{\al}_{p}-\xi)+\xi$,
  which is a contradiction because $\overline{\al}_{p}$ is a simple $\fr{t}$-root and it can not be expressed as the  sum of two positive $\fr{t}$-roots.
 Therefore   $\xi-\overline{\al}_{p}\notin R_{\fr{t}}$.  Let now $\al$ be a root  such that $\kappa(\al)=\xi$.  By assumption, it is $\xi\in R_{\fr{t}}^{+}$ and $\xi\notin\Pi_{\fr{t}}$. Since the projection $\kappa$ always maps roots from $\Pi_{M}$ to simple $\fr{t}$-roots  (see \cite{Chry2}), we conclude that it must be $\al\in R^{+}$ and $\al\not\in \Pi$.  However, we  proved that $\xi-\overline{\al}_{p}\notin R_{\fr{t}}$, or equivalently
 $
  \xi-\overline{\al}_{p}=\kappa(\al)-\kappa(\al_{p})=\kappa(\al-\al_{p})\notin R_{\fr{t}}.
 $
  Since $R_{K}=\emptyset$  and $\kappa(0)=0\notin R_{\fr{t}}$, where $0\in\fr{t}^{*}$, the last condition is true if and only if $\al=\al_{p}$, or $\al-\al_{p}\notin R$. The first condition is rejected by assumption.  The second one  is a contradiction due to Lemma \ref{carter}. This proves  our claim.
  \end{proof}
     \begin{remark}\label{bravo}
  \textnormal{Note that Theorem \ref{chrysik2} does not tell us anything about the uniqueness of the simple root $\al_{p}$, but just for its existence.  Thus to a $\fr{t}$-root $\xi\in R_{\fr{t}}$ (with $\xi\notin\Pi_{t}$), there may correspond  more that one simple roots $\al_{p}\in\Pi$ with $\xi-\overline{\al}_{p}\in R_{\fr{t}}^{+}$ (see Example \ref{su4}).} 
  \end{remark}
       \begin{corol}\label{fulchrysik2}
  Let $M=G/T$ be a full flag manifold of a compact simple Lie group $G$ with $\rnk G=\dim_{\bb{R}} T\geq 2$, and let $R_{\fr{t}}$ be the associated $\fr{t}$-root system.  Then  any $\fr{t}$-root $\xi\in R_{\fr{t}}$ such that $\xi\notin\Pi_{\fr{t}}$, belongs to a symmetric $\fr{t}$-triple, i.e.  we can always  find $\zeta, \eta\in R_{\fr{t}}$  such that    $\xi+\zeta+\eta=0$.
  \end{corol}
 \begin{proof}
 Let $\xi\in R_{\fr{t}}^{+}$.  By Theorem \ref{chrysik2}   we can find a suitable simple root $\al_{p}\in\Pi=\Pi_{M}$, such that $\xi-\overline{\al}_{p}\in R_{\fr{t}}^{+}$, where $\overline{\al}_{p}=\al_{p}|_{\fr{t}}\in\Pi_{\fr{t}}$. Then $(\overline{\al}_{p}, \xi-\overline{\al}_{p}, -\xi)$ is  the desired symmetric $\fr{t}$-triple,  since $\overline{\al}_{p}+ (\xi-\overline{\al}_{p}) + (-\xi)=0$, $\overline{\al}_{p}, \xi-\overline{\al}_{p}\in R_{\fr{t}}^{+}$, and $-\xi\in R_{\fr{t}}^{-}$.
  Thus  on $G/T$ we can determine symmetric $\fr{t}$-triples in the following way:
   \begin{equation}\label{fultriple1}
  \pm(\overline{\al}_{p}, \xi-\overline{\al}_{p}, -\xi), \quad \xi\in R_{\fr{t}}^{+}. 
     \end{equation}
 If $\xi$ is negative then    $-\xi\in R_{\fr{t}}^{+}$, and thus  in order to obtain the symmetric $\fr{t}$-triple which contains  $-\xi$ we have to replace in   (\ref{fultriple1})  $\xi$ by $-\xi$.  
      \end{proof}    
   Note  that other   inequivalent  symmetric $\fr{t}$-triples in the $\fr{t}$-root system $R_{\fr{t}}$ of  $G/T$, can be obtained  as follows:
   \begin{equation}\label{fultriple2}
   \pm \big(\eta, \zeta, -(\eta+\zeta)\big),  
   \end{equation}
   where $\eta, \zeta \in R_{\fr{t}}$ such that  $\eta+\zeta\in R_{\fr{t}}$,  but $\xi, \zeta\notin\Pi_{\fr{t}}$.

   \begin{corol}\label{fulchrysik3}
  For a full flag manifold $G/T$  of a compact simple Lie group $G$ with $\rnk G=\dim_{\bb{R}} T\geq 2$, there exist  at least $|R^{+}|-|\Pi|=|R^{+}|-\ell$ inequivallent symmetric $\fr{t}$-triples. 
   \end{corol}
   \begin{proof}
   This is an immediate consequence of Corrolary \ref{fulchrysik2}.
   \end{proof}
  
  \begin{remark}\label{geq2}
  \textnormal{Note that the condition $\rnk G=\dim_{\bb{R}} T\geq 2$ is assumed in order to avoid the extreme case of $\bb{C}P^{1}=\SU(2)/\U(1)=\SO(4)/\U(2)$, for which   we know that $\mathcal{S}=\emptyset$.   About  the first examples of full flag manifolds $G/T$, where $G$ is one of the  Lie groups $\SU(\ell+1), \SO(2\ell+1)$, $\Sp(\ell)$ with $\ell\geq 2$, or $G=\SO(2\ell)$,  with $\ell\geq 3$, we have $\mathcal{S}\neq \emptyset$.  In particular, by computing the associated $\fr{t}$-root systems, we  see that 
  \begin{center}
  \begin{tabular}{lll}
  $G$ & $G/T$  &  The non zero structure constants\\
  \thickline 
  $\SU(3)$ & $\SU(3)/T$ & $c_{12}^{2}$ (see  Remark \ref{3tropy}), \\
  $\SO(5)$ & $\SO(5)/T$ & $c_{12}^{3}, c_{23}^{4}$ (see  Remark \ref{4isotropy}), \\
  $\Sp(2)$ & $\Sp(2)/T$ & $c_{12}^{3}, c_{13}^{4}$ (see  Remark \ref{4isotropy}), \\
  $\SO(6)$ & $\SO(6)/T$ & $c_{12}^{4}, c_{13}^{5}, c_{25}^{6}, c_{34}^{6}$\\
  \hline
  \end{tabular}
  \end{center}}
  \end{remark}
  \begin{example}\label{su4}
  \textnormal{Consider the full flag manifold $\SU(4)/T$. Recall that a maximal torus $T$ of $\SU(4)$ is given by
  $
  T=\{\diag\{e^{i\theta_1}, \ldots, e^{i\theta_4}\} : \theta_{i}\in\bb{R}, \ \sum_{i}\theta_{i}=0\},
  $
with Lie  algebra    $\fr{t}=\{\diag\{i\theta_1, \ldots, i\theta_4\} : \theta_{i}\in\bb{R}, \ \sum_{i}\theta_{i}=0\}.
  $
     Thus the complexification  
     $
     \fr{t}^{\bb{C}}=\{\diag\{h_1, \ldots, h_4\} : h_{i}\in\bb{C}, \ \sum_{i}h_{i}=0\}, 
    $
    is a Cartan subalgebra of the complex simple Lie algebra $(\fr{su}(4))^{\bb{C}}=\fr{sl}_{4}\bb{C}$.   It is obvious that $\dim_{\bb{R}}\fr{t}=\dim_{\bb{C}}\fr{t}^{\bb{C}}=3=\rnk(\SU(4))$.  
      The root system of  $\fr{sl}_{4}\bb{C}$ with respect to $\fr{t}^{\bb{C}}$, is given by 
  $
  R=\{\pm(e_1-e_2), \pm(e_1-e_3), \pm(e_1-e_4), \pm(e_2-e_3), \pm(e_2-e_4), \pm(e_3-e_4)\},
  $
   and a basis of simple roots can be choosen as follows:  $\Pi=\{\al_1=e_1-e_2, \al_2=e_2-e_3, \al_3=e_3-e_4\}$.  With respect to $\Pi$ an ordering in $R$ is given by    $R^{+}=\{\al_1, \al_2, \al_3, \al_1+\al_2, \al_2+\al_3, \al_1+\al_2+\al_3\}$. The full flag $M=\SU(4)/T$ is obtained by setting $\Pi_{M}=\Pi$, and we have $R_{M}=R$ and $R_{M}^{+}=R^{+}$. Since $|\Pi_{M}|=3$, it is obvious that  $b_{2}(M)=3$.  A $\fr{t}$-basis is  given $\Pi_{\fr{t}}=\{\overline{\al}_1, \overline{\al}_2, \overline{\al}_3\}$ where $\overline{\al}_{i}=\al_{i}|_{\fr{t}}$, for any $i\in\{1, 2, 3\}$, and the positive $\fr{t}$-roots have the form
    $
   R_{\fr{t}}^{+}=\{\overline{\al}_1, \  \overline{\al}_2, \ \overline{\al}_3, \ \overline{\al}_1+\overline{\al}_2, \ \overline{\al}_2+\overline{\al}_3, \ \overline{\al}_1+\overline{\al}_2+\overline{\al}_3\}.
   $
   Thus $\fr{m}=\fr{m}_1\oplus\cdots\oplus\fr{m}_6$ (see also Table 3).  Related to Remark \ref{bravo}, note that the positive   $\fr{t}$-root $\xi=\overline{\al}_1+\overline{\al}_2+\overline{\al}_3$ belongs into two inequivalent symmetric $\fr{t}$-triples, since we can find two different simple roots $\al_{p}\neq \al_{q}$, such that $\xi-\overline{\al}_{p}\in R_{\fr{t}}^{+}$ and $\xi-\overline{\al}_{q}\in R_{\fr{t}}^{+}$. Indeed, we have $\xi-\overline{\al}_1=\overline{\al}_2+\overline{\al}_3\in R_{\fr{t}}^{+}$ and $\xi-\overline{\al}_3=\overline{\al}_1+\overline{\al}_2\in R_{\fr{t}}^{+}$. 
 Thus we obtain the symmetric $\fr{t}$-triples $\pm(\overline{\al}_1, \xi-\overline{\al}_1, -\xi)=\pm\big(\overline{\al}_1, \overline{\al}_2+\overline{\al}_3, -(\overline{\al}_1+\overline{\al}_2+\overline{\al}_3)\big)$, and $\pm(\overline{\al}_3, \xi-\overline{\al}_3, -\xi)=\pm\big(\overline{\al}_3, \overline{\al}_1+\overline{\al}_2, -(\overline{\al}_1+\overline{\al}_2+\overline{\al}_3\big)$.
   Also we can define the symmetric $\fr{t}$-triples $\pm\big(\overline{\al}_1, \overline{\al}_2, -(\overline{\al}_1+\overline{\al}_2)\big)$ and $\pm\big(\overline{\al}_2, \overline{\al}_3, -(\overline{\al}_2+\overline{\al}_3)\big)$. 
Thus the only non-zero structure constants are $c_{12}^{3}$, $c_{23}^{5}$, $c_{15}^{6}$, $c_{34}^{6}$, and their symmetries. }
  \end{example}

   \begin{remark}\label{corres}
   \textnormal{Due to Proposition \ref{GT} one  can  establish a bijective correspondence between
   symmetric $\fr{t}$-triples in the $\ell$-dimensional vector space $\fr{t}^{*}$ and triples of roots    with zero sum:
   \begin{equation}\label{onetoonee}
   \left\{
   \begin{tabular}{c}
   symmetric $\fr{t}$-triples \\
   $\overline{\al}+\overline{\be}+\overline{\gamma}=0$ \\
   $\overline{\al}, \overline{\be}, \overline{\gamma}\in R_{\fr{t}}$
   \end{tabular}\right\} \quad \Longleftrightarrow \quad
   \left\{
   \begin{tabular}{c}
   triples of roots with zero sum \\
   $\al+\be+\gamma=0$ \\
   $\al, \be, \gamma\in R$
   \end{tabular}\right\}
   \end{equation}
    \noindent  This correspondence shows that Theorem \ref{chrysik2} is the formulation of Lemma   \ref{carter}   in terms of $\fr{t}$-roots. Also (\ref{onetoonee}) make possible the rewritting of all  results about synmmetric $\fr{t}$-triples on full flag manifolds, just in terms of roots.  For axample, based on Proposition \ref{numberfull},  one can denote the triple   associated to the submodules $\fr{m}_{\al}, \fr{m}_{\be}$, and $\fr{m}_{\gamma}$ of $\fr{m}$, by $\displaystyle{ \gamma \brack \al \ \beta}$. Then}
  \end{remark}

   \begin{corol}\label{fullnew}
 Let $M=G/T$ be a full flag manifold of a compact simple Lie group \ $G$  and let   $\fr{m}=\sum_{\al\in R^{+}}\fr{m}_{\al}$ be the associated $( \cdot , \cdot )$-orthogonal decomposition of $\fr{m}$ into pairwise inequivalent irreducible $\ad(T)$-modules.  Then,  $\displaystyle{ \gamma \brack \al \ \be}\neq 0$, if and only if  $\al+\be-\gamma=0$.
   \end{corol}
  \begin{proof}
  This result is based on the correspondence (\ref{onetoonee}). Note that we have $\al+\be-\gamma=0$ instead of $\al+\be+\gamma=0$, since in the triple $\displaystyle { \gamma \brack \al \ \be}$  associated to the submodules $\fr{m}_{\al}, \fr{m}_{\be}$  and $\fr{m}_{\gamma}$, the roots $\al, \be$ and $\gamma$ are positive by assumption (see Proposition \ref{numberfull}).
  \end{proof}

\begin{example}\label{g2t}
\textnormal{Consider the full flag manifold $\G_2/T$. The root system of the exceptional Lie group $\G_2$  is given by $R=\{\mp(2e_2+e_3), \mp(e_2+2e_3), \pm(e_2-e_3), \mp(e_2+e_3), \mp e_2, \mp e_3\}$ (cf. \cite{AA}).  We fix a system of simple roots to be  $\Pi=\{\al_{1}=e_{2}-e_{3}, \al_{2}=-e_2\}$.  With repsect to $\Pi$ the positive roots are given by     \begin{equation}\label{positiveG2}
   R^{+}=\{\al_{1}, \al_{2}, \al_{1}+\al_{2}, \al_{1}+2\al_{2}, \al_{1}+3\al_{2}, 2\al_{1}+3\al_{2}\}.
   \end{equation}
    The  maximal root  is $\widetilde{a}=2\al_1+3\al_2$.    The angle between $\al_1$ and $\al_2$ is $5\pi/6$ and  we have $\left\|\al_1\right\|=\sqrt{3}\left\|\al_2\right\|$. We set $\Pi_{M}=\Pi$. It determines the full flag manifold $\G_{2}/T$. The set of complemetary roots $R_{M}$ coincide with the root system $R$,  i.e. $R=R_{M}$.  A $\fr{t}$-base is given by $\Pi_{\fr{t}}=\{\overline{\al}_{1}, \overline{\al}_{2}\}$, thus  for any $\al\in R^{+}$, we have $\xi=\kappa(\al)=c_{1}\overline{\al}_{1}+c_2\overline{\al}_{2}\in R_{\fr{t}}^{+}$, where $0\leq c_1\leq  2$ and $0\leq c_2\leq 3$. Thus, the positive $\fr{t}$-roots are given as follows: $R_{\fr{t}}^{+}=\{\xi_1=\overline{\al}_{1}, \ \xi_2=\overline{\al}_{2}, \ \xi_3=\overline{\al}_{1}+\overline{\al}_{2}, \ \xi_4=\overline{\al}_{1}+2\overline{\al}_{2}, \
\xi_5=\overline{\al}_{1}+3\overline{\al}_{2}, \ \xi_6=2\overline{\al}_{1}+3\overline{\al}_{2}\}$. 
 The isotropy representation $\fr{m}$ of $\G_2/T$ decomposes into six    inequivalent irreducible $\ad(\fr{t})$- submodules, that is 
   $\fr{m}=\fr{m}_1\oplus\fr{m}_2\oplus\fr{m}_3\oplus\fr{m}_4\oplus\fr{m}_5\oplus\fr{m}_6$, where any  isotropy summand  $\fr{m}_{\al}, (\al\in R^{+})$  admits the following expression:
 \begin{equation}\label{submodules}
 \left.
 \begin{tabular}{lll}
$\fr{m}_1= \mathbb{R}A_{\al_{1}}+\mathbb{R}B_{\al_1}$, & $\fr{m}_3= \mathbb{R}A_{\al_1+\al_2}+\mathbb{R}B_{\al_1+\al_2}$, &    $\fr{m}_5= \mathbb{R}A_{\al_1+3\al_2}+\mathbb{R}B_{\al_1+3\al_2}$\\
$\fr{m}_2= \mathbb{R}A_{\al_{2}}+\mathbb{R}B_{\al_2}$, & $\fr{m}_4=\mathbb{R}A_{\al_1+2\al_2}+\mathbb{R}B_{\al_1+2\al_2}$, &  
$\fr{m}_6=  \mathbb{R}A_{2\al_1+3\al_2}+\mathbb{R}B_{2\al_1+3\al_2}$.  
\end{tabular}\right\}
\end{equation}
}
\textnormal{Let now determine the non-zero structure constants  $\displaystyle {k \brack ij}$ of $\G_2/T$.   Due to Corollary \ref{fullnew}, it is sufficient to determine all triples of roots    $(\al, \be, \gamma)$ with zero sum, that is   $\al+\be+(-\gamma)=0$, for example.  By using  relations (\ref{positiveG2}) and (\ref{submodules}),  we get
\[
\begin{tabular}{ll}
   $\al_1+\al_2+(-(\al_1+\al_2))=0$, & $\al_1+(\al_1+3\al_2)+(-(2\al_1+3\al_2))=0$, \\ 
$\al_2+(\al_1+\al_2)+(-(\al_1+2\al_2))=0$,    & $(\al_1+\al_2)+(\al_1+2\al_2)+(-(2\al_1+3\al_2))=0$. \\   
$\al_2+(\al_1+2\al_2)+(-(\al_1+3\al_2))=0$, &  
\end{tabular}
\]
Thus, the only non-zero structure constants of $\G_2/T$ are the following (for their values, see \cite{ACS3}):
{\small{\[
 { \al_1+\al_2 \brack \al_1 \ \al_2}={ 3 \brack 1 2}, 
  { \al_1+2\al_2 \brack \al_2 \  \al_1+\al_2}={ 4 \brack 2 3}, 
  { \al_1+3\al_2 \brack \al_2 \  \al_1+2\al_2}={ 5 \brack 2 4}, 
   { 2\al_1+3\al_2 \brack \al_1 \  \al_1+3\al_2}={ 6 \brack 1 5}, 
 { 2\al_1+3\al_2 \brack \al_1+\al_2 \  \al_1+2\al_2}={ 6 \brack 34}.
  \] }} 
  }
\end{example}


 \markboth{Ioannis Chrysikos}{Flag manifolds, symmetric $\fr{t}$-triples and Einstein metrics}
 \section{Homogeneous Einstein metrics on  flag manifolds  with five isotropy summands}
 \markboth{Ioannis Chrysikos}{Flag manifolds, symmetric $\fr{t}$-triples and Einstein metrics}

Let us now proceed with the final part of the present work, the investigation of homogeneous Einstein metrics on flag manifolds  $G/K$ of the simple Lie group $G=\SO(7)$, whose isotropy representation decomposes into five pairwise inequivalent irreducible $K$-modules, i.e. $\fr{m}=\fr{m}_{1}\oplus\fr{m}_{2}\oplus\fr{m}_{3}\oplus\fr{m}_{4}\oplus\fr{m}_{5}$.   These spaces $M=\SO(7)/K$ are obtained by painting black two simple roots in the Dynkin diagram of $\SO(7)$.  This means that  $b_{2}(M)=2$, but the converse is not true, that is, there is also a flag manifold  $M=\SO(7)/K$ with $b_{2}(M)=2$, with four isotropy summands.   Indeed, by painting  black two simple roots in the Dynkin diagram of $\SO(7)$ we get  three possible cases:
\[
 \begin{picture}(50,20)(40,-8)
  \put(46, 0){\circle*{4}}
\put(46,8.5){\makebox(0,0){$\al_{1}$}}
   \put(48, 0.5){\line(1,0){14}}
\put(62, 0){\circle*{4}}
\put(62,8.5){\makebox(0,0){$\al_{2}$}}
\put(62,-10){\makebox(0,0){ $\bold{(I)}$}}
\put(63.8, 1.1){\line(1,0){15.5}}
\put(63.8, -0.7){\line(1,0){15.5}}
\put(76.5, -1.9){\scriptsize $>$}
\put(83.6, 0){\circle{4}}
 \put(84,8.5){\makebox(0,0){$\al_{3}$}}
\end{picture}, 
\quad   \mbox{or} \quad \begin{picture}(50,20)(40,-8)
  \put(46, 0){\circle*{4}}
\put(46,8.5){\makebox(0,0){$\al_{1}$}}
   \put(48, 0.5){\line(1,0){12}}
\put(62, 0){\circle{4}}
\put(62,8.5){\makebox(0,0){$\al_{2}$}}
\put(62,-10){\makebox(0,0){ $\bold{(II)}$}}
\put(63.8, 1.1){\line(1,0){15.5}}
\put(63.8, -0.7){\line(1,0){15.5}}
\put(76.5, -1.9){\scriptsize $>$}
\put(83.6, 0){\circle*{4}}
 \put(84,8.5){\makebox(0,0){$\al_{3}$}}
\end{picture}, 
\quad    \mbox{or} \quad \begin{picture}(50,20)(40,-8)
  \put(46, 0){\circle{4}}
\put(46,8.5){\makebox(0,0){$\al_{1}$}}
   \put(48, 0.5){\line(1,0){14}}
\put(62, 0){\circle*{4}}
\put(62,8.5){\makebox(0,0){$\al_{2}$}}
\put(62,-10){\makebox(0,0){ $\bold{(III)}$}}
\put(63.8, 1.1){\line(1,0){15.5}}
\put(63.8, -0.7){\line(1,0){15.5}}
\put(76.5, -1.9){\scriptsize $>$}
\put(83.6, 0){\circle*{4}}
 \put(84,8.5){\makebox(0,0){$\al_{3}$}}
\end{picture}.
\]
The  painted Dynkin diagrams (I), (II) and (III) determine  the same coset $\SO(7)/K=\SO(7)/\U(1)^{2}\times \SU(2)$, with the difference that $\SU(2)$ is generated by the simple roots $\al_3$, $\al_2$ and $\al_1$, respectively.  The flag manifold which is defined by the PDD (I), is given also by $\SO(7)/\U(1)^{2}\times \SO(3)$ and it  belongs to the family $\SO(2\ell+1)/\U(1)^{2}\times \SO(2\ell-3)$, which has four isotropy summands  \cite[Prop. 5]{Chry2}. Homogeneous Einstein metrics on this manifold  have been classified in \cite[Theor. 6]{Chry2}. The  flag space  which is defined by the PDD (II)    is also presented by $\SO(7)/\U(1)\times \U(2)$ and it  belongs to  the family $\SO(2\ell+1)/\U(1)\times \U(\ell-1)$.  This  is   defined by the subset $\Pi_{M}=\{\al_1, \al_{\ell}\}\subset\Pi$ of a system of simple roots $\Pi=\{\al_1, \ldots, \al_{\ell}\}$ corresponding to $\SO(2\ell+1)$, and it  has five isotropy summands \cite[Prop. 6]{Chry2}.  The   flag manifold  defined by the PDD (III)  has also five isotropy summands  and it belongs to the family $\SO(2\ell+1)/\U(1)^{2}\times \SU(\ell-1)$, which is defined by the set $\Pi_{M}=\{\al_{\ell-1}, \al_{\ell}\}\subset\Pi$.   
 
\subsection{The congruence of the flag manifolds defined by the PDD (II) and (III)} 
  Next we will see that the flag manifolds $\SO(7)/\U(1)\times \U(2)$ and $\SO(7)/\U(1)^{2}\times \SU(2)$,  defined by the subsets $\Pi_{M}=\{\al_1, \al_3\}$ and $\Pi_{M}=\{\al_2, \al_3\}$, respectively, apart from  diffeomorphic they are also isometric (as real manifolds). First we   fix some notation.     Let $R=\{\pm e_{i}\pm e_{j} : 1\leq i\neq j\leq 3\}\cup\{\pm e_{i} : 1\leq i\leq 3\}
$
 be the root system of   $\fr{g}^{\bb{C}}=\fr{so}(7, \bb{C})$.  We fix, once and for all, a system of simple roots  
$
\Pi=\{\al_1=e_1-e_2, \al_2=e_2-e_3, \al_3=e_3\}
$
  for  $R$, and we denote by 
  $
  R^{+}=\{\al_1, \al_2, \al_3, \al_1+\al_2, \al_2+\al_3, \al_2+2\al_3, \al_1+\al_2+\al_3,  \al_1+\al_2+2\al_3, \al_1+2\al_2+2\al_3\}
 $
   the corresponding set of positive roots. Note that  $\varphi(\al_1, \al_1)=\varphi(\al_2, \al_2)=2$, $\varphi(\al_3, \al_3)=1$, $\varphi(\al_1, \al_2)=\varphi(\al_2, \al_3)=-1$, and $\varphi(\al_1, \al_3)=0$  (cf. \cite{Cart}).
   Recall also that    the roots of   $\fr{so}(7, \bb{C})$ are divided into two classes with repect to their length, namely  the long roots $\{\pm e_{i}\pm e_{j} : 1\leq i\neq j\leq 3\}$, and  the short roots    $\{\pm e_{i} : 1\leq i\leq 3\}$ (cf. \cite[p.~ 147]{Cart}). For convenience we set
\begin{equation}\label{ls}
\left.\begin{tabular}{rll}
  positive long roots &:& $\mathcal{L}=\{\al_1, \al_2, \al_1+\al_2, \al_2+2\al_3, \al_1+\al_2+2\al_3, \al_1+2\al_2+2\al_3\}$ \\
 positive short roots  &:& $\mathcal{S}=\{\al_3, \al_2+\al_3, \al_1+\al_2+\al_3\}$  
 \end{tabular}\right\} 
 \end{equation} 
 In Table 4   we present the most important feautures of the flag manifolds $\SO(7)/K$ defined by the PDD (II) and (III). 
   In part (a),    we describe the   root system  $R_{K}$ of the  semismple part $\fr{su}(2, \bb{C})$ of   $\fr{k}^{\bb{C}}=\fr{su}(2, \bb{C})\oplus \bb{C}\oplus\bb{C}$, the associated positive complementary roots $R_{M}^{+}$, a $\fr{t}$-basis $\Pi_{\fr{t}}$, and the corresponding positive $\fr{t}$-root system $R_{\fr{t}}^{+}$. In part (b)  we state the reductive decomposition   of   $\fr{g}=T_{e}\SO(7)$ with respect to  the $\Ad(K)$-invariant inner product $( \cdot , \cdot )=-\varphi( \cdot , \cdot )$, and we apply relation (\ref{bas}) to  determine the corresponding  irreducible $\ad(\fr{k})$-sumbodules. In part (c)  we present the corresponding symmetric $\fr{t}$-triples and  finally,  in part (d),     we   give the associated  non-zero structure constants $c_{ij}^{k}$.
  \begin{center}
 {\bf Table 4.} \ {\small The flag manifolds of $\SO(7)$ with five isotropy summands}
 \end{center}
  {\small{\begin{center}
\begin{tabular}{l|l|l}
& PDD (II) $\qquad  \begin{picture}(50,5)(20,-2)
  \put(46, 0){\circle*{4}}
   \put(48, 0.5){\line(1,0){12}}
\put(62, 0){\circle{4}}
\put(63.8, 1.1){\line(1,0){15.5}}
\put(63.8, -0.7){\line(1,0){15.5}}
\put(76.5, -1.9){\scriptsize $>$}
\put(83.6, 0){\circle*{4}}
\end{picture}$ & PDD (III) $\qquad \begin{picture}(50,5)(20,-2)
  \put(46, 0){\circle{4}}
   \put(48, 0.5){\line(1,0){14}}
\put(62, 0){\circle*{4}}
\put(63.8, 1.1){\line(1,0){15.5}}
\put(63.8, -0.7){\line(1,0){15.5}}
\put(76.5, -1.9){\scriptsize $>$}
\put(83.6, 0){\circle*{4}}
\end{picture}$ \\
\hline 
 (a) & $M=\SO(7)/\U(1)\times \U(2)$, $\fr{m}=T_{o}M$  &  $M=\SO(7)/\U(1)^{2}\times \SU(2)$, $\fr{n}=T_{o}M$ \\
 \thickline 
&  $\Pi_{M}=\{\al_1, \al_3\}, \Pi_{K}=\{\al_2\}, R_{K}=\{\pm\al_2\}$ & $\Pi_{M}=\{\al_2, \al_3\}, \Pi_{K}=\{\al_{1}\}, R_{K}=\{\pm\al_1\}$ \\
& $R_{M}^{+}= \{\al_1, \al_3, \al_1+\al_2, \al_2+\al_3, \al_2+2\al_3, $ & $R_{M}^{+}=\{\al_2, \al_3, \al_1+\al_2, \al_2+\al_3, \al_2+2\al_3, $ \\
  & \ \ \ \ \   $\al_1+\al_2+\al_3,  \al_1+\al_2+2\al_3, \al_1+2\al_2+2\al_3\}$ & \ \ \ \ \  $\al_1+\al_2+\al_3,  \al_1+\al_2+2\al_3, \al_1+2\al_2+2\al_3\}$\\ 
 & $\Pi_{\fr{t}}=\{\overline{\al}_{1}=\al_{1}|_{\fr{t}}, \overline{\al}_{3}=\al_{3}|_{\fr{t}}\}$ &      $\Pi_{\fr{t}}=\{\overline{\al}_{2}=\al_{2}|_{\fr{t}}, \overline{\al}_{3}=\al_{3}|_{\fr{t}}\}$  \\
&  $R_{\fr{t}}^{+}=\{\overline{\al}_{1}, \overline{\al}_{3}, \overline{\al}_{1}+\overline{\al}_{3}, 2\overline{\al}_{3}, \overline{\al}_{1}+2\overline{\al}_{3}\}$ & $R_{\fr{t}}^{+}=\{\overline{\al}_{2}, \overline{\al}_{3}, \overline{\al}_{2}+\overline{\al}_{3}, \overline{\al}_{2}+2\overline{\al}_{3}, 2\overline{\al}_{2}+2\overline{\al}_{3}\}$ \\
 \hline
 (b) &  $\fr{m}=\fr{m}_1\oplus\fr{m}_2\oplus\fr{m}_3\oplus\fr{m}_4\oplus\fr{m}_{5}$ &  $\fr{n}=\fr{n}_1\oplus\fr{n}_2\oplus\fr{n}_3\oplus\fr{n}_4\oplus\fr{n}_{5}$ \\
  \thickline 
 & $\fr{m}_1=  \displaystyle\sum_{\al\in R_{M}^{+} : \kappa(\al)=\overline{\al}_{1}}(\bb{R}A_{\al}+\bb{R}B_{\al})$ & $\fr{n}_1=  \displaystyle\sum_{\al\in R_{M}^{+} : \kappa(\al)=\overline{\al}_{2}}(\bb{R}A_{\al}+\bb{R}B_{\al})$ \\
 & $\fr{m}_2=  \displaystyle\sum_{\al\in R_{M}^{+} : \kappa(\al)=\overline{\al}_{3}}(\bb{R}A_{\al}+\bb{R}B_{\al})$ & $\fr{n}_2=  \displaystyle\sum_{\al\in R_{M}^{+} : \kappa(\al)=\overline{\al}_{3}}(\bb{R}A_{\al}+\bb{R}B_{\al})$ \\
 & $\fr{m}_3=  \displaystyle  \sum_{\al\in R_{M}^{+} : \kappa(\al)=\overline{\al}_{1}+\overline{\al}_{3}}(\bb{R}A_{\al}+\bb{R}B_{\al})$ & $\fr{n}_3=  \displaystyle  \sum_{\al\in R_{M}^{+} : \kappa(\al)=\overline{\al}_{2}+\overline{\al}_{3}}(\bb{R}A_{\al}+\bb{R}B_{\al})$ \\
 & $\fr{m}_4=  \displaystyle \sum_{\al\in R_{M}^{+} : \kappa(\al)=2\overline{\al}_{3}}(\bb{R}A_{\al}+\bb{R}B_{\al})$ &  $\fr{n}_4=  \displaystyle \sum_{\al\in R_{M}^{+} : \kappa(\al)=\overline{\al}_{2}+2\overline{\al}_{3}}(\bb{R}A_{\al}+\bb{R}B_{\al})$ \\
 & $\fr{m}_5=  \displaystyle \sum_{\al\in R_{M}^{+} : \kappa(\al)=\overline{\al}_{1}+2\overline{\al}_{3}}(\bb{R}A_{\al}+\bb{R}B_{\al})$ & $\fr{n}_5=  \displaystyle \sum_{\al\in R_{M}^{+} : \kappa(\al)=2\overline{\al}_{2}+2\overline{\al}_{3}}(\bb{R}A_{\al}+\bb{R}B_{\al})$ \\
  \hline 
 (c) & symmetric $\fr{t}$-triples & symmetric $\fr{t}$-triples \\
 \thickline 
 &  $\big(\overline{\al}_1, \overline{\al}_3, -(\overline{\al}_1+\overline{\al}_3)\big)$, & $\big(\overline{\al}_2, \overline{\al}_3, -(\overline{\al}_2+\overline{\al}_3)\big)$,   \\
 &  $\big(\overline{\al}_3, \overline{\al}_3, - 2\overline{\al}_3\big)$, & $\big(\overline{\al}_3, \overline{\al}_2+\overline{\al}_3, -(\overline{\al}_2+2\overline{\al}_3)\big)$, \\
 & $\big(\overline{\al}_1,   2\overline{\al}_3, -(\overline{\al}_1+2\overline{\al}_3)\big)$,  & $\big(\overline{\al}_2, \overline{\al}_2+2\overline{\al}_3, -(2\overline{\al}_2+2\overline{\al}_3)\big)$,  \\
 &  $\big(\overline{\al}_3, \overline{\al}_1+\overline{\al}_3, -(\overline{\al}_1+2\overline{\al}_3)\big)$ & $\big(\overline{\al}_2+\overline{\al}_3, \overline{\al}_2+\overline{\al}_3, -(2\overline{\al}_2+2\overline{\al}_3)\big)$  \\
 \hline 
 (d) & non-zero structure constants & non-zero structure constants \\
 \thickline  
 &  $c_{12}^{3}$, $c_{22}^{4}$, $c_{14}^{5}$, $c_{23}^{5}$ & $c_{12}^{3}$, $c_{23}^{4}$, $c_{14}^{5}$, $c_{33}^{5}$ \\ 
\hline
 \end{tabular}
 \end{center} }}
 \vskip 0.5cm
 \begin{center}
 {\bf Table 5.} \ {\small The isotropy summands and their  type  with respect to the length of  roots}
 \end{center} 
   {\small{ \begin{center}
  \begin{tabular}{ l|l|c}
    isotropy decomposition $\fr{m}=\fr{m}_1\oplus\fr{m}_{2}\oplus\fr{m}_{3}\oplus\fr{m}_{4}\oplus\fr{m}_{5}$ &  type or roots &   $d_{i}=\dim\fr{m}_{i}$  \\
\thickline
    $\fr{m}_{1}=\Span\{A_{\al_1}+B_{\al_{1}}, A_{\al_1+\al_2}+B_{\al_1+\al_2}\}$ & long & 4\\
   $\fr{m}_{2}=\Span\{A_{\al_3}+B_{\al_{3}}, A_{\al_2+\al_3}+B_{\al_2+\al_3}\}$ & short & 4 \\
  $\fr{m}_{3}=\Span\{A_{\al_1+\al_2+\al_3}+B_{\al_{1}+\al_{2}+\al_{3}}\}$ & short & 2 \\ 
  $\fr{m}_{4}=\Span\{A_{\al_2+2\al_3}+B_{\al_{2}+2\al_{3}}\}$ & long & 2 \\ 
  $\fr{m}_{5}=\Span\{A_{\al_1+\al_2+2\al_3}+B_{\al_1+\al_{2}+2\al_{3}},   A_{\al_1+2\al_2+2\al_3}+B_{\al_1+2\al_{2}+2\al_{3}}\}$ & long & 4 \\
\hline 
  isotropy decomposition $\fr{n}=\fr{n}_1\oplus\fr{n}_{2}\oplus\fr{n}_{3}\oplus\fr{n}_{4}\oplus\fr{n}_{5}$ &   type of roots &  $d_{i}'=\dim\fr{n}_{i}$  \\
\thickline
  $\fr{n}_{1}=\Span\{A_{\al_2}+B_{\al_{2}}, A_{\al_1+\al_2}+B_{\al_1+\al_2}\}$ & long & 4 \\
   $\fr{n}_{2}=\Span\{A_{\al_{3}}+B_{\al_{3}}\}$ & short & 2 \\
  $\fr{n}_{3}=\Span\{A_{\al_2+\al_3}+B_{\al_2+\al_3}, A_{\al_{1}+\al_2+\al_3}+B_{\al_{1}+\al_{2}+\al_{3}}\}$ & short & 4 \\ 
  $\fr{n}_{4}=\Span\{A_{\al_1+\al_2+2\al_3}+B_{\al_1+\al_{2}+2\al_{3}}, A_{\al_2+2\al_3}+B_{\al_{2}+2\al_{3}}\}$ & long  & 4\\ 
  $\fr{n}_{5}=\Span\{A_{\al_1+2\al_2+2\al_3}+B_{\al_1+2\al_{2}+2\al_{3}}\}$ & long & 2 
\end{tabular} \end{center}}}

   \begin{prop}\label{desired}
The flag manifolds $\SO(7)/\U(1)\times \U(2)\cong \SO(7)/\U(1)^{2}\times \SU(2)$ defined by the painted Dynkin diagrams (II) and (III), respectively, are isometric (as real manifolds). 
\end{prop}
\begin{proof}   
 By combining parts (a) and (b) of Table 4  and by using relation (\ref{ls}),   we have established  in Table 5 a correspondence between the associated  isotropy summands and the type or roots   generating them (with respect to  their length).  We also give the dimensions of these submodules. From the first and the third column  of Table 5, it obvious that  in order to prove our claim is sufficient to construct the following isometries:
\[
 \sigma:   \fr{m}_{1}\to \fr{n}_{1}, \ \ \     \sigma:  \fr{m}_{2}\to  \fr{n}_{3},  \ \ \  \sigma: \fr{m}_{3}\to \fr{n}_{2}, \ \ \
  \sigma: \fr{m}_{4}\to \fr{n}_{5},  \ \ \
  \sigma: \fr{m}_{5}\to \fr{n}_{4}.
 \]
 Then we will have establish the desired isometry $\sigma :\fr{m}\to\fr{n}$.   This map $\sigma :\fr{m}\to\fr{n}$ is  induced by   the Weyl group $\mathcal{W}$  of $\fr{so}(7, \bb{C})$, which acts on the root system $R$  via the reflections  $\{s_{\al} : \al \in R\}$, given by  
 $
 s_{\al}(\be)=\be-\displaystyle\frac{2\varphi(\be, \al)}{\varphi(\al, \al)}\al,  
 $
 for any $ \al, \be\in R$.
In particular, since the irreducible  modules of the following pairs   $(\fr{m}_{1}, \fr{n}_{1})$, $(\fr{m}_{2}, \fr{n}_{3})$, $(\fr{m}_{3}, \fr{n}_{2})$, $(\fr{m}_{4}, \fr{n}_{5})$, and $(\fr{m}_{5}, \fr{n}_{4})$, are generated  in any case   by root vectors corresponding to roots of the same length,  the existence of such an isometry between them,  is a subsequence of the  transitive action of $\mathcal{W}$   on  roots of a given length. This  means, that if we have a pair $(\fr{m}_{i}, \fr{n}_{j})$ with  $\dim\fr{m}_{i}=\dim\fr{n}_{j}$  and we assume for example  that the modules $\fr{m}_{i}$ and $\fr{n}_{j}$ are generating   by the vectors $A_\al+B_{\al}$ and $A_{\be}+B_{\be}$, respectively,  where both $\al, \be$ are such that $\al, \be\in \mathcal{L}$, then we   can always  find an element $w\in\mathcal{W}$ such that $w(\al)=\be$ (or $w(\al)=-\be)$. Since an element of $\mathcal{W}$ induces inner automorphism of our Lie algebra,  the  element $w\in\mathcal{W}$   will determine an isometry $w : \fr{m}
_{i}\to\fr{n}_{j}$ (the same is true for the short roots, too).  Indeed, recall   that the group $\mathcal{W}$ is generated by the simple reflections  $\{s_{1}=s_{\al_{1}}=s_{e_1-e_2}, s_{2}=s_{\al_{2}}=s_{e_2-e_3}, s_{3}=s_{\al_{3}}=s_{e_3}\}$.  For these reflections it is   $s_{\al_{i}}(\al_{j})=s_{i}(\al_{j})=\al_{j}- b_{ij}\al_{i}$, where $B=(b_{ij})$ is the transpose of the Cartan matrix   of  $\fr{so}(7, \bb{C})$, given by  
$
 A=(A_{ij})=\left(2\varphi(\al_{i}, \al_{j})/\varphi(\al_{j}, \al_{j})\right)=\begin{pmatrix}
 2 & -1 & 0 \\
 -1 & 2 & -2 \\
 0 & -1 & 2 
 \end{pmatrix}$. 
 Thus, it is 
  $s_{1}(\al_{1})=-\al_{1}$,  $s_{2}(\al_{1})=\al_1+\al_2$,  $s_{3}(\al_{1})=\al_{1}$, 
  $s_1(\al_{2})=\al_1+\al_2$, $s_{2}(\al_{2})=-\al_2$, $s_{3}(\al_{2})=\al_{2}+2\al_3$,
  $s_{1}(\al_{3})=\al_{3}$, $s_{2}(\al_{3})=\al_2+\al_3$, $s_{3}(\al_{3})=-\al_{3}$  and   we easily compute  that
 {\small{\begin{eqnarray}
  (s_{2}\circ s_{1})(\al_1) &=&  -(\al_1+\al_2) , \label{s1}\\
  (s_{2}\circ s_{1})(\al_1+\al_2) &=&  -\al_2, \label{s2}\\
  (s_{2}\circ s_{1})(\al_3) &=&  \al_2+\al_3, \label{s3}\\
   (s_{2}\circ s_{1})(\al_2+\al_3) &=&  \al_1+\al_2+\al_3, \label{s4}\\ 
  (s_{2}\circ s_{1})(\al_1+\al_2+\al_3) &=&  \al_3, \label{s5}\\
  (s_{2}\circ s_{1})(\al_2+2\al_3)  &=& \al_1+2\al_2+2\al_3, \label{s6}\\
  (s_{2}\circ s_{1})(\al_1+\al_2+2\al_3) &=& \al_2+2\al_3, \label{s7}\\
 (s_{2}\circ s_{1})(\al_1+2\al_2+2\al_3) &=&  \al_1+\al_2+2\al_3. \label{s8}
 \end{eqnarray}}}
The relations  (\ref{s1})-(\ref{s2})  show that $\fr{m}_{1}, \fr{n}_{1}$ are isometric via the reflection $s_{2}\circ s_{1} :\fr{m}_{1}\to\fr{n}_{1}$.  Similar, relations  (\ref{s3})-(\ref{s4})  ensure that  $s_{2}\circ s_{1}$ maps   $\fr{m}_{2}\to\fr{n}_{3}$, relation  (\ref{s5})   implies that $s_{2}\circ s_{1}: \fr{m}_{3}\to\fr{n}_{2}$, and relation  (\ref{s6})  implies that $s_{2}\circ s_{1}: \fr{m}_{4}\to\fr{n}_{5}$. Finally, by  (\ref{s7})-(\ref{s8})  we  conclude that $\fr{m}_{5}, \fr{n}_{4}$ are also isometric via the composition  $s_{2}\circ s_{1}$ and in this way we have define an isometry  $\sigma=s_{2}\circ s_{1} : \fr{m}\to\fr{n}$.      
       Note that under the above isometry $s_{2}\circ s_{1} : \fr{m}\to\fr{n}$,  the symmetric $\fr{t}$-triples and hence the structure constants of the flag manifolds defined by the PDD (II) and (III), are identified.
 \end{proof}

 \subsection{Homogeneous Einstein metrics on $\bold{SO(7)/U(1)\times U(2)\cong SO(7)/U(1)^{2}\times SU(2)}$}
For the construction of the Einstein equation for an $\SO(7)$-invariant Riemannian metric we need to recall the explicit formulae of the Ricci tensor and the scalar curvature for a flag manifold $M=G/K$ in the general case.  

Let $\fr{g}=\fr{k}\oplus\fr{m}$ be a reductive decomposition of $\fr{g}$ with respect to  $( \cdot , \cdot )=-\varphi( \cdot , \cdot)$.  As we have said before,  a  $G$-invariant Riemannian metric $g$ on $M=G/K$ is identified with an $\Ad(K)$-invariant (or $\ad(\fr{k})$-invariant) inner product  $\langle \cdot , \cdot \rangle$ on $\fr{m}$. This inner product can be written by $\langle X, Y\rangle=(AX, Y)$ $(X, Y\in\fr{m})$ for some $\Ad(K)$-invariant positive definite symmetric endomorphism $A : \fr{m}\to\fr{m}$.  Due to Proposition \ref{isotropy},  we can express  $A$ by the equation
$
A=\sum_{\xi\in R_{\fr{t}}^{+}}x_{\xi}\cdot \Id\big|_{(\fr{m}_{\xi}\oplus\fr{m}_{-\xi})^{\tau}},
$
where each element $\{x_{\xi} : \xi\in R_{\fr{t}}^{+}\}$ is an eigenvalue (a positive real number) of $A$,  with associated eigenspace the real $\ad(\fr{k})$-module  $(\fr{m}_{\xi}\oplus\fr{m}_{-\xi})^{\tau}$.    If we assume for simplicity that $\fr{m}=\oplus_{i=1}^{s}\fr{m}_{i}$ is a  $( \cdot , \cdot )$-orthogonal decomposition of $\fr{m}$ into $s$ pairwise inequivalent irreducible $\ad(\fr{k})$-modules $\fr{m}_{i}$ which are given by (\ref{bas}), then  $A$ is given by $A=\sum_{\xi_{i}\in R^{+}_{\fr{t}}}x_{\xi_{i}}\cdot \Id\big|_{\fr{m}_{i}}=\sum_{i=1}^{s}x_{i}\cdot \Id\big|_{\fr{m}_{i}}$, where $x_{i}\equiv x_{\xi_{i}}$ for any $\xi_{i}\in R_{\fr{t}}^{+}=\{\xi_1, \ldots, \xi_{s}\}$.
   Due to the decomposition $\fr{m}_{i}=\sum_{\al\in R_{M}^+ \ :\ \kappa(\al)=\xi_{i}}(\mathbb{R}A_{\al}+\mathbb{R}B_{\al})$, it is obvious that the vectors $\{A_{\al}, B_{\al} : \al\in R_{M}^{+}\}$ are eigenvectors of $A$, corresponding to the eigenvalue $x_{i}\equiv x_{\xi_{i}}$, and thus we also denote this eigenvalue by $x_{\al}\in \bb{R}^{+}$, where $\al\in R_{M}^+$ is such that $\kappa(\al)=\xi_{i}$, for any $1\leq i\leq s$. As usual, we   extend $A$ to a complex linear operator  in $\fr{m}^{\bb{C}}$ without any change in notation. Hence  the inner product $g=\langle \cdot, \cdot \rangle$ admits a natural extension to an $\ad(\fr{k}^{\bb{C}})$-invariant bilinear symmetric form     on $\fr{m}^{\bb{C}}$, and for this one we maintain  the same notation too.  Then,  the   root vectors $\{E_{\al} : \al\in R_{M}\}$ are eigenvectors of   $A : \fr{m}^{\bb{C}} \to \fr{m}^{\bb{C}}$ corresponding to the eigenvalues $x_{\al}=x_{-\al}>0$.  Note that   $x_{\al}=x_{\be}$  
whenever $\al|_{\fr{t}}=\be|_{\fr{t}}$, for any  $\al, \be\in R_{M}^{+}$.   By the $\ad(\fr{k}^{\bb{C}})$-invariance of the metric and the part (ii) of the definition of an invarinat ordering,  we also obtain   $x_{\al}=x_{\al+\be}$ for any $\al, \al+\be\in R_{M}^{+}$ with $\be\in R_{K}$.  

The set of $G$-invariant Riemannian metrics on  $G/K$ is paramatrized by $|R_{\fr{t}}^{+}|$ real numbers, it means that  any $G$-invariant Riemannian metric  $g=\langle \cdot, \cdot \rangle=(A\cdot, \cdot)$ on $G/K$ is given by 
{\small{
\[
g= \langle \cdot, \cdot \rangle = (A\cdot , \cdot)= x_1\cdot ( \cdot , \cdot )|_{\fr{m}_1}+\cdots+x_s\cdot ( \cdot , \cdot )|_{\fr{m}_s},
\]}}
where   $x_1\equiv x_{\xi_{1}}>0, \ldots, x_s \equiv x_{\xi_{s}}>0$.   Next, we will denote such an invariant Riemannian metric by $g=(x_1, \ldots, x_{s})\in\bb{R}_{+}^{s}$.  Note that given a $G$-invariant complex structure $J$ on $M=G/K$  (such a structure is induced by an invariant ordering $R_{M}^{+}$ and it is determined by an $\Ad(K)$-invariant endomorphism $J_{o} : \fr{m}^{\bb{C}}\to\fr{m}^{\bb{C}}$ such that $J_{o}E_{\pm a}=\pm i E_{\pm \al}$ for any $\al\in R_{M}^{+}$),  a $G$-invariant   metric   is K\"ahler with respect to $J$, if and only if the positive real numbers $x_{\xi}$ satisfy the equation $x_{\xi+\zeta}=x_{\xi}+x_{\zeta}$ for any $\xi, \zeta, \xi+\zeta\in R_{\fr{t}}^{+}=\kappa(R_{M}^{+})$. In other words, $g$ is K\"ahler, if and only if $x_{\al+\be}=x_{\al}+x_{\be}$, where $\al, \be, \al+\be\in R_{M}^{+}$ are such that $\kappa(\al)=\xi$ and $\kappa(\be)=\zeta$ (cf. \cite{Ale1}).   

 In view of the decomposition $\fr{m}=\oplus_{k=1}^{s}\fr{m}_{k}$,   the Ricci tensor $\Ric_{g}$ of $(G/K, g=(x_1, \ldots, x_{s}))$, as a $G$-invariant symmetric  covariant  2-tensor on $G/K$, is identified with an $\Ad(K)$-invariant symmetric bilinear form  on $\fr{m}$ and it is given by  $\Ric_{g} = \sum_{k=1}^{s}y_k\cdot ( \cdot , \cdot )|_{\fr{m}_k}$, 
for some $y_{1}, \ldots, y_{s}\in\bb{R}$.     Let $\{e_{j}^{(k)}\}_{j=1}^{d_{k}}$  be an $(\cdot , \cdot)$-orthonormal basis on $\fr{m}_{k}$, where $d_{k}=\dim \fr{m}_{k}$,  for any $k\in\{1, \ldots, s\}$.  Then the set  $\{X_{j}^{(k)}=1/\sqrt{x_{k}}e_{j}^{(k)}\}$ is an $\langle \cdot, \cdot\rangle$-orthonormal basis of $\fr{m}_{k}$, where   $\langle \cdot, \cdot\rangle$ is the $\Ad(K)$-invariant inner product   induced by the $G$-invariant metric $g=(x_1, \ldots, x_s)$ of $M$.   Let $r_{k}$ be the real numbers defined by the equation
    $
    r_k =\Ric_{g}(X_{j}^{(k)}, X_{j}^{(k)})=\Ric_{g}(\frac{e_j^{(k)}}{\sqrt{x_k}}, \frac{e_j^{(k)}}{\sqrt{x_k}}) = (1/x_{k})\Ric_{g} (e_j^{(k)}, e_j^{(k)}),
   $
that is  $\Ric_{g} (e_j^{(k)}, e_j^{(k)}) = x_k r_k$.  Then it obvious that  we can express the Ricci tensor by  $\Ric_{g}=\sum_{k=1}^{s}(x_{k} r_{k})\cdot ( \cdot, \cdot)|_{\fr{m}_{k}}$.  In particular, we have that 
    \begin{prop}\label{Ricc}{\textnormal{(\cite{SP}, \cite{Wa2})}} 
      The components $r_{k}$ of the Ricci tensor of $(G/K, g=(x_1, \ldots, x_{s}))$    are given by
  $
  \displaystyle
   r_{k}=\frac{1}{2x_{k}}+\frac{1}{4d_{k}}\sum_{i, j}\frac{x_{k}}{x_{i}x_{j}}  {k \brack ij}-\frac{1}{2d_{k}}\sum_{i, j}\frac{x_{j}}{x_{k}x_{i}}   {j \brack ki},
   $
    for any $k=1, \ldots, s$.  Moreover, the scalar curvature $S_{g}$ of $(G/K, g=(x_1, \ldots, x_{s}))$ has the form
$
  S_{g}={\rm tr}\Ric_{g}=\sum_{k=1}^{s}d_{k}r_{k}=\displaystyle\frac{1}{2}\sum_{k=1}^{s}\frac{d_{k}}{x_{k}}-\frac{1}{4}\sum_{i, j, k} {k  \brack ij}\frac{x_{k}}{x_{i}x_{j}}$.
 \end{prop} 
 A $G$-invariant metric  on $M=G/K$ is   Einstein  with Einstein constant $\lambda\in \bb{R}_{+}$, if and only if, it is a positive real solution of the  system $\{r_1=\lambda, \ r_2=\lambda, \ \ldots, \ r_{k}=\lambda\}\Leftrightarrow\{r_1-r_2=0, \ r_2-r_3=0, \ \ldots, \ r_{k-1}-r_{k}=0\}$.
 
 Let us now focus again on the flag manifold  $M=\SO(7)/\U(1)\times \U(2)\cong \SO(7)/\U(1)^{2}\times \SU(2)$.  Due to the   isometry $\sigma :\fr{m}\to\fr{n}$ presented in \textsection 3.1,  we will not distinguish these $\SO(7)$-flag spaces.  Without loss of generality, we assume that $M$ is defined  for example by the PDD (II), that   is  $\Pi_{M}=\{\al_1, \al_3\}$ and $\fr{m}=\fr{m}_1\oplus\fr{m}_2\oplus\fr{m}_3\oplus\fr{m}_4\oplus\fr{m}_{5}$. 
   We consider  $\SO(7)$-invariant Riemannian metrics $g=\langle \cdot , \cdot \rangle$ on $M$, given by
 $
 g=\langle \cdot , \cdot  \rangle=x_1\cdot ( \cdot , \cdot )|_{\fr{m}_1}+x_2\cdot ( \cdot , \cdot )|_{\fr{m}_2}+x_3\cdot ( \cdot , \cdot )|_{\fr{m}_3}+x_4\cdot ( \cdot , \cdot )|_{\fr{m}_4}+x_5\cdot ( \cdot , \cdot )|_{\fr{m}_5},
$
 where $(x_1, x_2, x_3, x_4, x_5)\in\mathbb{R}_{+}^{5}$.     The Ricci tensor $\Ric_{g}$ of $(M, g)$, with repsect to an $\langle \cdot, \cdot\rangle$-orthonormal basis of $\fr{m}$    is given by 
 $\Ric_{g}=\sum_{k=1}^{5}(x_{k}r_{k})\cdot ( \cdot , \cdot )|_{\fr{m}_{k}}$, where the components $r_{k}$ are defined by  Proposition \ref{Ricc}.
By using Table 5    we easily get:
      \begin{prop}\label{components1}
  The  components 
  $r_{k}$ of the Ricci tensor  $\Ric_{g}$ on $(M, g=(x_1, x_2, x_3, x_4, x_5))$ are  given by
\begin{equation}\label{riccomp1}
  \ \  \left.
\begin{tabular}{ll}
$r_1=$ & $\displaystyle\frac{1}{2x_1}+  \frac{c_{12}^3}{2d_1}\Big( \frac{x_1}{x_2x_3}- \frac{x_2}{x_1x_3}- \frac{x_3}{x_1x_2}\Big)+\displaystyle\frac{c_{14}^5}{2d_1}\Big( \frac{x_1}{x_4x_5}- \frac{x_4}{x_1x_5}- \frac{x_5}{x_1x_4}\Big)$, \\
$r_2=$ & $\displaystyle\frac{1}{2x_2}+  \frac{c_{12}^3}{2d_2}\Big( \frac{x_2}{x_1x_3}- \frac{x_1}{x_2x_3}- \frac{x_3}{x_1x_2}\Big)+  \frac{c_{23}^5}{2d_2}\Big( \frac{x_2}{x_3x_5}- \frac{x_3}{x_2x_5}-\frac{x_5}{x_2x_3}\Big)-\displaystyle\frac{c_{22}^{4}}{2d_{2}}\frac{x_{4}}{x_{2}^{2}}$, \\
$r_3=$ & $\displaystyle\frac{1}{2x_3}+  \frac{c_{12}^3}{2d_3}\Big( \frac{x_3}{x_1x_2}- \frac{x_2}{x_1x_3}- \frac{x_1}{x_2x_3}\Big)+  \frac{c_{23}^5}{2d_3}\Big( \frac{x_3}{x_2x_5}- \frac{x_2}{x_3x_5}- \frac{x_5}{x_2x_3}\Big)$, \\
$r_4=$ & $\displaystyle\frac{1}{2x_4}+  \frac{c_{14}^5}{2d_4}\Big( \frac{x_4}{x_1x_5}- \frac{x_1}{x_4x_5}- \frac{x_5}{x_1x_4}\Big)+
\displaystyle\frac{c_{22}^{4}}{4d_4}\Big(\frac{x_4}{x_{2}^{2}}-\frac{2}{x_{4}}\Big),$ \\
$r_5=$ & $\displaystyle\frac{1}{2x_5}+  \frac{c_{14}^5}{2d_5}\Big( \frac{x_5}{x_1x_4}- \frac{x_1}{x_4x_5}- \frac{x_4}{x_1x_5}\Big)+  \frac{c_{23}^5}{2d_5}\Big( \frac{x_5}{x_2x_3}- \frac{x_2}{x_3x_5}- \frac{x_3}{x_2x_5}\Big).$  
\end{tabular}\right\} 
\end{equation}
  \end{prop}
 
     For the computation of the non zero triples $c_{12}^{3}, c_{14}^{5}, c_{22}^{4}$ and $c_{23}^{5}$ we will use a K\"ahler-Einstein metric of $M$. For a detailed description of invariant complex structures and Kahler-Einstein metrics see \cite{AP} or  \cite{Chry2}. Here we only recall that following well-known result:
\begin{theorem}\label{MainKE}\textnormal{(\cite{Ale1}, \cite{AP}, \cite{Chry2})}
 Let $J$ be the $G$-invariant 
complex structure   on $M$ defined by the 
invariant ordering $R_{M}^{+}$ in $R_{M}$.    Then, the   $\ad(\fr{k})^{\bb{C}}$-invariant Riemannian 
  metric on $\fr{m}^{\bb{C}}$ given by $g_{J}=\{x_{\al}=c\cdot \varphi(\delta_{\fr{m}}, \al) \ (c\in\mathbb{R}) : \al\in R_{M}^+\}$, where $\delta_{\fr{m}}=\frac{1}{2}\sum_{\be\in R_{M}^+}\be$,   is a K\"ahler-Einstein metric (up to  a constant) on $M$  compatible with $J$.  
     \end{theorem}
    The weight $\delta_{\fr{m}}$ is called the {\it Koszul form}. If $M=G/K$ is defined by a set $\Pi_{M}=\{\al_{i_{1}}, \ldots, \al_{i_{r}}\}$, then we have $2\delta_{\fr{m}}=c_{i_{1}}\Lambda_{i_{1}}+\cdots +c_{i_{r}}\Lambda_{i_{r}}$,  
where $\Lambda_{i_{1}}, \ldots, \Lambda_{i_{r}}$ 
are the fundamental weights   corresponding to the elements of $\Pi_{M}$, and $c_{i_{1}}, \ldots, c_{i_{r}}
\in\bb{Z}_{+}$.
  \begin{prop}\label{ke1}
 Let $J$ be the $G$-invariant complex structure on $M$ induced from the invariant ordering $R_{M}^{+}$ given in Table 4.  Then, the $\SO(7)$-invariant K\"ahler-Einstein metric $g_{J}$ corresponding to $J$ is given   (up to a scale)  by $g=(3, 2, 5, 4, 7)$. 
 \end{prop}
 \begin{proof}
 First we    compute the Koszul form $\delta_{\fr{m}}$ corresponding to the invariant ordering $R_{M}^{+}$ given in Table 4.  Since our flag manifold $M$ is such that $\Pi_{M}=\{\al_{1}, \al_{3}\}$, it is $2\delta_{\fr{m}}=c_1\Lambda_{1}+c_{3}\Lambda_{3}$, where the positive numbers $c_{1}$ and $c_{3}$ are under investigation.  Based on the relation $R_{M}^{+}=R^{+}\backslash R_{K}^{+}$ we easily obtain $2\delta_{\fr{m}}=2\delta_{G}-2\delta_{K}=2(\Lambda_1+\Lambda_2+\Lambda_3)- \al_{2}$.   
Morever, by using  the relation $\al_{i}=\sum_{j=1}^{3}A_{ij}\Lambda_{j}$, where $(A_{ij})$ are the entries of the Cartan matrix, we have $\al_{2}= -\Lambda_{1}+2\Lambda_{2}-2\Lambda_3$, thus $2\delta_{\fr{m}}= 3\Lambda_{1}+4\Lambda_{3}$ and $\delta_{\fr{m}}=3/2\Lambda_{1}+2\Lambda_{3}$.  
 According to Theorem \ref{MainKE}, the K\"ahler-Einstein metric $g_{J}$ 
which is compatible   to the natural invariant complex structure $J$ defined by the invariant  ordering $R^+_{M}$, is given by
 $
g_{J}=x_{\overline{\al}_1}\cdot (  \cdot , \cdot )|_{\fr{m}_1}+x_{\overline{\al}_3}\cdot ( \cdot , \cdot )|_{\fr{m}_2}+x_{\overline{\al}_1+\overline{\al}_3}\cdot ( \cdot , \cdot )|_{\fr{m}_3}
+x_{2\overline{\al}_3}\cdot ( \cdot , \cdot )|_{\fr{m}_4}+x_{\overline{\al}_1+2\overline{\al}_3}\cdot ( \cdot , \cdot )|_{\fr{m}_5},
$
where  the positive numbers $x_{\xi_k}$ are given by $x_{\xi_k}=\varphi(\delta_{\fr{m}}, \al)$  (here the root $\al\in R_{M}^{+}$ is such that $\al\in \kappa^{-1}(\xi_{k})$, and $\xi_{k}$ is the associated $\fr{t}$-root of $\fr{m}_{k}$). By an easy computation  we obtain the following values:
 {\small{\begin{eqnarray*}
x_{\overline{\al}_1} &=& \varphi(3/2\Lambda_{1}+2\Lambda_{3}, \al_1)= \frac{3\varphi(\Lambda_1, \al_1)}{2}=\frac{3\varphi(\al_1, \al_1)}{4}=\frac{3\varphi(\al_3, \al_3)}{2}, \\
x_{\overline{\al}_3} &=&  \varphi(3/2\Lambda_{1}+2\Lambda_{3}, \al_3)= 2\varphi(\Lambda_3, \al_3)= \varphi(\al_3, \al_3), \\
x_{\overline{\al}_1+\overline{\al}_3} &=&  \varphi(3/2\Lambda_{1}+2\Lambda_{3}, \al_1+\al_2+\al_3)= \frac{3\varphi(\Lambda_1, \al_1)}{2}+2\varphi(\Lambda_3, \al_3) = \frac{5\varphi(\al_{3}, \al_{3})}{2},  \\
x_{2\overline{\al}_3} &=&  \varphi(3/2\Lambda_{1}+2\Lambda_{3}, \al_2+2\al_3)= 4\varphi(\Lambda_3, \al_3)  =2\varphi(\al_3, \al_3), \\
x_{\overline{\al}_1+2\overline{\al}_3} &=&  \varphi(3/2\Lambda_{1}+2\Lambda_{3}, \al_1+2\al_2+2\al_3)=  \frac{3\varphi(\Lambda_1, \al_1)}{2}+4\varphi(\Lambda_3, \al_3)  =\frac{7\varphi(\al_{3}, \al_{3})}{2}.
\end{eqnarray*} }}
By substituting the value $\varphi(\al_{3}, \al_{3})=1$,  and after a normalization,  the result follows.  
  \end{proof}
 
 \begin{remark}\label{a2a3}\textnormal{
  Note that for the flag manifold $\SO(7)/K$ defined by the PDD (III), for the invariant ordering $R_{M}^{+}$ given by in Table 4, we find that $\delta_{\fr{m}}=3/2\Lambda_{2}+\Lambda_{3}$.  The corresponding $\SO(7)$-invariant K\"ahler-Einstein metric $g_{J}$, is given by      (up to a scale)   $g_{J}=(3, 1, 4, 5, 8)$.} 
 \end{remark}

Let us now consider  the system 
 \begin{equation}\label{syss1}
 r_1-r_2=0, \quad r_2-r_3=0, \quad  r_3-r_4=0, \quad r_4-r_5=0,
 \end{equation}
  where the components $r_1, r_2, r_3, r_4$ and $r_5$ are given by (\ref{riccomp1}). In system (\ref{syss1}) we substitute the dimensions  $d_{i}=\dim\fr{m}_{i}$ presented in Table 5  and the coefficients $x_1=3$, $x_2=2$, $x_3=5$, $x_4=4$ and $x_5=7$ of the K\"ahelr-Einstein metric $g_{J}$. Then we obtain 
 \begin{lemma}\label{sc22}
The structure  constants $c_{12}^{3}, c_{14}^{5}, c_{22}^{4}$ and $c_{23}^{5}$ of $M=\SO(7)/\U(1)\times \U(2)$ with respect to the decomposition  $\fr{m}=\fr{m}_1\oplus\fr{m}_2\oplus\fr{m}_3\oplus\fr{m}_4\oplus\fr{m}_{5}$  are  given by $c_{12}^{3}=c_{14}^{5}=c_{22}^{4}=c_{23}^{5}=2/5$.
 \end{lemma}

By  Lemma \ref{sc22}, Proposition \ref{components1} and Table 5,   system (\ref{syss1}) reduces to the following   system (we also apply  the normalization $x_1=1$):
  \begin{equation}\label{eq1}
{{\left.\begin{tabular}{r}
$\displaystyle\frac{ x_3 x_4^2 x_5 - x_2^3 (x_4 + 2 x_4 x_5) - 
  x_2^2 x_3 (-1 + x_4^2 - 10 x_4 x_5 + x_5^2) + 
  x_2 x_4 (x_3^2 - 10 x_3 x_5 + x_5 (2 + x_5))}{20 x_2^2 x_3 x_4 x_5}=0$ \\
  $\displaystyle\frac{-10 x_2^2 x_5 - x_3 x_4 x_5 + 3 x_2^3 (1 + x_5) + 
  x_2 (10 x_3 x_5 - 3 x_3^2 (1 + x_5) + x_5 (1 + x_5))}{20 x_2^2 x_3 x_5}=0$ \\
 $\displaystyle\frac{-x_3 x_4^2 x_5 - 2 x_2^3 x_4 (1 + x_5) + 2 x_2 x_4 (x_3^2 - x_5) (1 + x_5) + 
 2 x_2^2 (5 x_4 x_5 + x_3 (1 - x_4^2 - 4 x_5 + x_5^2))}{20 x_2^2 x_3 x_4 x_5}=0$ \\
 $\displaystyle\frac{x_2^3 x_4 + x_3 x_4^2 x_5 + 
  x_2^2 x_3 (-1 - 10 x_4 + 3 x_4^2 + 8 x_5 - 3 x_5^2) + x_2 x_4 (x_3^2 - x_5^2)}{20 x_2^2 x_3 x_4 x_5}=0$
\end{tabular}\right\}}}
\end{equation}
Any  positive real solution   $x_2>0, x_3>0, x_4>0, x_5>0$  of   system (\ref{eq1}), determines an  $\SO(7)$-invariant Einstein metric $(1, x_2, x_3, x_4, x_5)\in\bb{R}^{5}_{+}$ on $M$. One can obtain all these solutions  by applying  for example, the command {\ttfamily{NSolve}} in Mathematica.

\begin{theorem}\label{Themain1}
 The flag manifold $M=\SO(7)/ \U(1)\times \U(2)\cong \SO(7)/\U(1)^{2}\times \SU(2)$ admits (up to a scale) eight $\SO(7)$-invariant Einstein metrics, which approximately are given as follows
{\small{\[ 
  \begin{tabular}{ll}
   $(a) \ (1, 1.0231, 0.3089, 1.8751, 0.9999),$ & $(b) \ (1, 1.0157, 0.2458,  0.5319,  0.9999),$ \\
   $(c) \ (1, 0.5422, 0.9898,  0.5176, 0.6571)$, & $(d)  (1,  0.8251,  1.5063,  0.7877, 1.5217),$ \\
   $(e) \  (1, 4/5, 1/5, 8/5, 3/5),$ & $(f) \ (1, 4/3, 1/3, 8/3, 5/3),$ \\
   $(g) \ (1, 2/3, 5/3, 4/3, 7/3),$ & $(h) \ (1,  2/7, 5/7, 4/7, 3/7).$
   \end{tabular} \]}}
The metrics $(e), (f), (g)$ and $(h)$ are K\"ahler-Einstein.  
\end{theorem}

 \subsection{The isometry problem for the Einstein metrics}
 
 We will examine now the isometry problem for the homogeneous Einstein metrics 
stated in Theorem  \ref{Themain1}. We follow the method presented in \cite{Chry2}.

 Let $M=G/K$ be a generalized flag manifold with 
 $\fr{m}=\oplus_{i=1}^{5}\fr{m}_i$, $d_{i}=\dim \fr{m}_{i}$, and  $d=\sum_{i=1}^{5}d_{i}=\dim M$. 
 For any $G$-invariant Einstein metric $g=(x_1, x_2, x_3, x_4, x_5)$ on  $M$ 
 we determine a normalized scale invariant given by
 $H_{g}=V_{g}^{1/d}S_g$, where  $S_g$ is the scalar curvature of $g$, 
 and   $V_{g}=\prod_{i=1}^{5}x_{i}^{d_i}$  is the volume of 
 of the given metric $g$.
Since, the scalar curvature is a homogeneous  polynomial of degree $-1$  on the variables $x_{i}$ (see Proposition \ref{Ricc}), and 
 the volume $V_{g}$ is a  monomial of degree $d$,
 the quantity  $H_{g}=V_{g}^{1/d}S_g$ is a homogeneous polynomial of degree 0. Therefore,
 $H_{g}$  is invariant under a common scaling of the variables $x_i$.  If two metrics are isometric then they have the same scale invariant, so if
  the scale invariants $H_{g}$ and $H_{g'}$ 
 are different, then the metrics $g$ and $g'$ can not  be isometric. 
But if $H_{g}=H_{g'}$,
 we cannot immediately conclude if the metrics $g$ and $g'$ are isometric or not. 
  For such a case we have to look at the   group of automorphisms of $G$ and check if there is an automorphism 
  which permutes the isotopy summands and takes one metric to another.  
     K\"ahler-Einstein metrics which correspond to 
   equivalent invariant complex structures on $M$ are isometric.

  Consider  our quotient $M=\SO(7)/\U(1)\times \U(2)\cong \SO(7)/\U(1)^{2}\times \SU(2)$, and let $g=(x_1, x_2, x_3, x_4, x_5)$ be an  $\SO(7)$-invariant Riemannian metric on $M$.  By applying   Propsition \ref{Ricc}, we easily find that the scalar curvature $S_{g}$   is given by
 {\small{\[
  S_{g} =\sum_{k=1}^{5}\frac{d_{k}}{x_{k}} -\frac{c_{12}^{3}}{4} \Big( \frac{x_1}{x_2x_3}+ \frac{x_2}{x_1x_3}+ \frac{x_3}{x_1x_2}\Big)- \frac{c_{14}^5}{4}\Big( \frac{x_1}{x_4x_5}+ \frac{x_4}{x_1x_5}+\frac{x_5}{x_1x_4}\Big)  
   -\frac{c_{23}^5}{4}\Big( \frac{x_2}{x_3x_5}+\frac{x_3}{x_2x_5}+ \frac{x_5}{x_2x_3}\Big) 
   -\frac{c_{22}^{4}}{4}\Big(\frac{x_4}{x_{2}^{2}}+\frac{2}{x_{4}}\Big).
  \]}}
 Thus, by using Lemma \ref{sc22} and the dimensions $d_{i}=\dim\fr{m}_{i}$ of Table 5, we get that  for a normalized $\SO(7)$-invariant metric  $g=(1, x_2, x_3, x_4, x_5)$, the scale invariant $H_{g}$ is given by 
  \begin{eqnarray*}
 H_{g} &=& \frac{-  x_2^2 x_3 x_4 x_5^3}{10 (x_2^4 x_3^2 x_4^2 x_5^4)^{\frac{15}{16}}}
\Big(x_3 x_4^2 x_5 + 2 x_2^3 x_4 (1 + x_5) + 
    2 x_2 x_4 \big(-10 x_3 x_5 + x_3^2 (1 + x_5) + x_5 (1 + x_5)\big) \\
    &&  +     2 x_2^2 \big(-5 x_4 x_5 + x_3 (1 + x_4^2 - 4 x_5 + x_5^2 - 10 x_4 (1 + x_5))\big)\Big).
 \end{eqnarray*}

    For the $\SO(7)$-invariant Einstein metrics presented in Theorem \ref{Themain1} we obtain  the following approximate values of the scale invariant
    $H_{g}$.
   
  \smallskip 
 \begin{center}
{\bf{Table 6.}} {\small{The values of  $H_{g}$ for the Einstein metrics which admits  $M=\SO(7)/\U(1)\times \U(2)\cong \SO(7)/\U(1)^{2}\times \SU(2)$}}
\end{center}  
\smallskip
  \begin{center}
\begin{tabular}{c||c|c|c|c|c|c|c|c}
 $\mbox{Invariant Einstein metrics}$ & $(a)$ & $(b)$ & $(c)$ & $(d)$ & $(e)$ & $(f)$ & $(g)$ & $(h)$\\
\thickline
 $\mbox{Corresponding values of} \ H_{g}$ & 5.7677 & 5.6670 & 5.8968  & 5.8968 & 5.7748  & 5.7748 & 5.9232  & 5.9232 
 \end{tabular}
\end{center}
  
  \medskip
Relatively to  Theorem  \ref{Themain1}, and in  view of   Table 6, we easily deduce  that the K\"ahler-Einstein metrics $(e)$ and  $(f)$ can not be isometric with the K\"ahler-Einstein metrics $(g)$ and $(h)$. However, by \cite[Theorem. 5]{Nis} we know   $M=\SO(7)/\U(1)\times \U(2)\cong \SO(7)/\U(1)^{2}\times \SU(2)$ admits  two pairs of equivalent $\SO(7)$-invariant complex structures (or equivalently, exactly two inequivalent complex structures), and thus, there are two pairs of isometric K\"ahler--Einstein metrics.  Since $H_{(e)}=H_{(f)}$ and $H_{(g)}=H_{(h)}$, we have that
\begin{corol}\label{isometry1}
The K\"ahler-Einstein metrics $(e)$ and $(f)$, presented in  Theorem \ref{Themain1}, are isometric each other and the same is true for the pair   $(g)$ and $(h)$, of the same theorem. Thus $M$   admits  precisely  two (up to a scale)  non isometric $\SO(7)$-invariant K\"ahler-Einstein metrics. 
\end{corol}
 Let now exam the isometry problem of the non-K\"ahler-Einstein metrics of $M$. The first two  metrics $(a)$ and $(b)$ presented in Theorem \ref{Themain1} are non-isometric each other, since $H_{(a)}\neq H_{(b)}$, and  for the same reason   they also cannot be  isometric  with any of the metrics $(c)$ and $(d)$.  In particular for the later Einstein metrics  we obtain the same scale invariant $H_{(c)}=H_{(d)}$, thus we are not still able to    conclude immediately  if these metrics are isometric or not.  However, we see that the metric $(c)$, given by $(x_1 =1,   x_2 = 0.54221,  x_3 = 0.98988,   
 x_4 = 0.51767,    x_5 = 0.65715 )$ is obtained from the metric $(d)$, i.e. the metric $(x_1 = 1,   x_2 = 0.82510,   x_3 =1.50633,   
 x_4 = 0.78776,    x_5 =1.52173)$,  by 
 dividing  the components of the later metric  with $1.52173= x5$ and interchanging $x_1$ and 
 $x_5$.  Thus, it is sufficient to find    a linear isomorphism $\sigma : { \frak m}_1 \to { \frak m}_5$, $\sigma : { \frak m}_2 \to { \frak m}_2$, $\sigma : { \frak m}_3 \to { \frak m}_3$ and  $\sigma : { \frak m}_4 \to { \frak m}_4$ which induces an exchange of the variables ${x_1}$and ${x_5}$ and remain fixed for the others.  This isometry is induced by the   action of the Weyl group $\mathcal{W}=\{s_{\al} : \al\in R\}$ of $R$. Indeed, 
 recall that for any $\al, \be\in R$ we have 
 \begin{equation}\label{refl}
 s_{\al}(\be)=\be-\displaystyle\frac{2\varphi(\be, \al)}{\varphi(\al, \al)}\al=\be-(p-q)\al,  
\end{equation}
  where the positive integers $p, q$ are determined by the $\al$-chain of roots through $\be$:
  $
 -p\al+\be, \ldots, \be, \ldots, q\al+\be.
 $
By applying relation (\ref{refl}), one can easily see that  
  the reflection $s_{\alpha_2 + 2 \alpha_3}$ by the root $\alpha_2 + 2 \alpha_3$ is such that 
\begin{eqnarray}
 s_{\alpha_2 + 2 \alpha_3}(\alpha_1) &=& \alpha_1 + \alpha_2 + 2 \alpha_3,  \label{r1}\\
  s_{\alpha_2 + 2 \alpha_3}(\alpha_2) &=&  \alpha_2,  \label{r2}\\
  s_{\alpha_2 + 2 \alpha_3}(\alpha_3) &=& -( \alpha_2 +  \alpha_3), \label{r3} \\ 
  s_{\al_2+2\al_3}(\al_1+\al_2) &=&  \al_1+2\al_2+2\al_3, \label{r4}\\
   s_{\al_2+2\al_3}(\al_2+\al_3) &=& -\al_3, \label{r5}\\
   s_{\al_2+2\al_3}(\al_1+\al_2+\al_3) &=&  \al_1+\al_2+\al_3,  \label{r6}\\
   s_{\al_2+2\al_3}(\al_2+2\al_3) &=&  -(\al_2+2\al_3),  \label{r7}\\
   s_{\al_2+2\al_3}(\al_1+\al_2+2\al_3) &=&  \al_1,  \label{r8}\\
   s_{\al_2+2\al_3}(\al_1+2\al_2+\al_3) &=&  \al_1+\al_2  \label{r9} 
   \end{eqnarray}
In view of the first column of Table  5, and due to relations (\ref{r1})-(\ref{r4}) and (\ref{r8})-(\ref{r9}), we conclude that the isometry $s_{\al_{2}+2\al_3}$ maps $\fr{m}_{1}\to\fr{m}_{5}$ and vice-versa.  Similarly,  relations (\ref{r3})-(\ref{r5}), (\ref{r6}), and (\ref{r7}),   imply that the reflection $s_{\al_{2}+2\al_{3}}$ induces a linear map which maps  the isotropy summands $\fr{m}_{2}$, $\fr{m}_{3}$ and $\fr{m}_{4}$   onto themselves  and leave them invariant. Thus the reflection $\sigma=s_{\al_2+2\al_3}$ shows that the metrics $(c)$ and $(d)$ are mutually isometric.

 \begin{corol}\label{isometry2}
The non-K\"ahler-Einstein metrics $(c)$ and $(d)$, presented in Theorem \ref{Themain1}, are isometric each other.    Thus   $M$  admits  precisely three (up to scale)    non-isometric   $\SO(7)$-invariant Einstein metrics, which are not K\"ahler with respect to any $\SO(7)$-invariant complex structure of $M$. 
\end{corol} 
Theorem  A   in introduction,   is following  now by Proposition \ref{desired}, Theorem \ref{Themain1} and  Corollaries  \ref{isometry1} and \ref{isometry2}.
 

\end{document}